\documentclass[11pt,reqno]{amsart}
\usepackage{a4wide}

\usepackage{mathrsfs}
\usepackage{amsfonts}
\usepackage{amsmath}
\usepackage{stmaryrd}
\usepackage{amssymb}
\usepackage{amsthm}
\usepackage{esint}

\usepackage{url}

\usepackage{amscd}
\usepackage{indentfirst}
\usepackage{enumerate}
\usepackage{graphicx}
\usepackage{appendix}

\usepackage[numbers,sort&compress]{natbib}

\usepackage{color}
\usepackage[colorlinks, linkcolor=blue, citecolor=blue]{hyperref}

\emergencystretch=\maxdimen
\newcommand{\pa}{\partial}

\renewcommand{\div}{{\rm div}}
\newcommand{\curl}{{\rm curl}}
\newcommand{\norm}[1]{\|#1\|}

\newtheorem{thm}{Theorem}[section]
\newtheorem{lem}[thm]{Lemma}

\newtheorem{cor}[thm]{Corollary}
\newtheorem{rem}[thm]{Remark}
\newtheorem{defi}[thm]{Definition}
\numberwithin{equation}{section}

\allowdisplaybreaks

\begin{document}

\title[Weak Serrin-type blowup criterion for the 3D full CNS]
{Weak Serrin-type criterion for the 3D full compressible Navier-Stokes equations}
\author[M.H. Xie]{Minghong Xie}
\address[Minghong Xie]{ Center for Applied Mathematics of Guangxi, Guangxi Normal University,
Guilin, Guangxi 541004, People's Republic of China}
\email{xiemh0622@hotmail.com}
\author[S.G. Xu]{Saiguo Xu}
\address[Saiguo Xu]{Center for Applied Mathematics of Guangxi, Guangxi Normal University,
Guilin, Guangxi 541004, People's Republic of China}
\email{xsgsxx@126.com}
\author[Y.H. Zhang]{Yinghui Zhang*}
\address[Yinghui Zhang] {Center for Applied Mathematics of Guangxi, Guangxi Normal University,
Guilin, Guangxi 541004, People's Republic of China} \email{yinghuizhang@mailbox.gxnu.edu.cn}

\thanks{* Corresponding author.}

\thanks{This work was supported by National Natural Science
Foundation of China $\#$12271114, Guangxi Natural Science Foundation $\#$2024GXNSFDA010071, $\#$2019JJG110003, $\#$2019AC20214, Science and Technology Project of Guangxi $\#$GuikeAD21220114, the Innovation Project of Guangxi Graduate Education $\#$JGY2023061, and the Key Laboratory of Mathematical Model and Application (Guangxi Normal University), Education Department of Guangxi Zhuang Autonomous Region.}

\date{\today}

\begin{abstract}
We investigate weak Serrin-type blowup criterion of the three-dimensional full compressible Navier-Stokes equations for the Cauchy problem, Dirichlet problem and Navier-slip boundary condition. It is shown that  the
strong or smooth solution exists globally if the density is bounded from above, and either the absolute
temperature
or velocity satisfies the weak Serrin's condition. Therefore, if the weak Serrin norm of the absolute
temperature or the velocity remains bounded, it is not possible for other
kinds of singularities (such as vacuum states vanish or vacuum appears in the non-vacuum region
or even milder singularities) to form before the density becomes unbounded. In particular, this criterion extends
those Serrin-type blowup criterion  results in (Math. Ann. 390 (2024): 1201-1248; Arch. Ration. Mech. Anal. 207(2013): 303-316).
Furthermore, as a by-product, for the isentropic compressible Navier-Stokes
equations, we succeed in removing the technical assumption $\rho_0\in L^1$ in (J. Lond. Math. Soc. (2) 102(2020): 125--142).
 The initial data can be arbitrarily large and allow to contain vacuum states here.
\end{abstract}

\maketitle

{\small

\section{Introduction}
The motion of a general viscous compressible, heat-conductive, ideal polytropic fluid in a domain $\Omega\subset\mathbb{R}^3$ is governed by the following full compressible Navier-Stokes equations
\begin{equation}\label{FCNS-eq}
  \begin{cases}
    \pa_t\rho+\div(\rho u)=0, \\
    \pa_t(\rho u)+\div(\rho u\otimes u)-\mu\Delta u-(\mu+\lambda)\nabla\div u+\nabla P(\rho)=0,\\
    c_v[\pa_t(\rho\theta)+\div(\rho u\theta)]-\kappa\Delta\theta+P\div u=2\mu|\mathcal{D}(u)|^2+\lambda(\div u)^2,
  \end{cases}
\end{equation}
where $\rho$, $u$, $\theta$, $P=R\rho\theta$ ($R>0$) denote the density, velocity, absolute temperature and pressure respectively, $\mu$ and $\lambda$ represent the shear viscosity and bulk viscosity coefficients satisfying the physical restrictions:
\begin{equation}\label{viscosity-condition}
  \mu>0,\quad \lambda+\frac{2}{3}\mu\geq0.
\end{equation}
The positive constants $c_v$ and $k$ are the heat capacity and the ratio of the heat conductivity coefficient over the heat capacity respectively.
In addition, $\mathcal{D}(u)$ is the deformation tensor
\begin{equation*}
  \mathcal{D}(u)=\frac{1}{2}(\nabla u+\nabla u^t).
\end{equation*}
The system \eqref{FCNS-eq} will be equipped with initial data
\begin{equation}\label{initial-data}
  (\rho, u, \theta)(x,0)=(\rho_0, u_0, \theta_0)(x), \quad x\in\Omega,
\end{equation}
and three types of boundary conditions:
\begin{itemize}
  \item Cauchy problem:
  \begin{equation}\label{Cauchy-condition}
    \Omega=\mathbb{R}^3~\textrm{and }(\rho,u,\theta)\rightarrow(\tilde{\rho},0,\tilde{\theta}),~\hbox{as}~|x|\rightarrow\infty,
  \end{equation}
  \noindent where constants $\tilde{\rho}$, $\tilde{\theta}\geq0$;
  \item Dirichlet problem: $\Omega$ is a bounded smooth domain in $\mathbb{R}^3$, and
      \begin{equation}\label{Dirichlet-condition}
        u|_{\pa\Omega}=0,\,\frac{\pa\theta}{\pa n}|_{\pa\Omega}=0;
      \end{equation}
  \item Navier-slip boundary condition: $\Omega\subset\mathbb{R}^3$ is a bounded and simply connected smooth domain, and
      \begin{equation}\label{Navier-slip-condition}
        u\cdot n=0,\, \curl u\times n=0 \textrm{ on }\pa\Omega,\quad \frac{\pa\theta}{\pa n}|_{\pa\Omega}=0,
      \end{equation}
      \noindent where $n$ is the unit outer normal to $\pa\Omega$.

\end{itemize}

There are large amounts of literature on the large time existence and behavior of solutions to the full
compressible Navier-Stokes equations \eqref{FCNS-eq}. For the one dimensional problem,
when the initial density and temperature are strictly positive, the existence,
uniqueness and regularity of (weak, strong or smooth)
solutions have been investigated extensively by many people. We refer to \cite{Kazhikhov1977,Kazhikhov1982} and references therein. The local well-posedness of multi-dimensional problem was studied by Nash \cite{Nash1958, Nash1962} and Serrin \cite{Serrin1959} in the absence of vacuum.
When the initial density is allowed to vanish, the existence and uniqueness of local strong solution was obtained in \cite{Cho2006-1}.
Matsumura and Nishida \cite{Matsumura1980,Matsumura1983} first obtained the global existence of strong solutions with initial data close to the non-vacuum equilibrium. Later, Hoff \cite{Hoff1995,Hoff1997} established the global weak solutions with strictly positive initial density and temperature for discontinuous initial data.
When initial vacuum states are allowed, the global existence and uniqueness of classical (or strong) solutions have been obtained by Huang and Li \cite{Huang-Li2018} for small initial energy and possible large oscillations, and Wen and Zhu \cite{Wen2017} for small initial mass, respectively. Recently, Li \cite{Li-JK2020} showed the global well-posedness of strong solutions under the assumption that certain scaling invariant quantity is small.
For the case that $\tilde{\rho}=0$ and $\tilde{\theta}=0$, Xin \cite{Xin1998} first proved that when the initial
density has compact support, there is no smooth solution to the Cauchy problem of the
full compressible Navier-Stokes equations \eqref{FCNS-eq} without heat conduction. See also the  generalization to the case that the initial data have an isolated mass group \cite{Xin2013}.

Therefore, it is important to study the mechanism of blowup and the structure of posssible singularities of strong (or smooth) solutions to the full compressible
Navier-Stokes equations \eqref{FCNS-eq}  with the initial data allowed to vanish.

 In the following, let us briefly review some former results closely related to the blowup criterion of the system \eqref{FCNS-eq}.

\begin{itemize}
  \item Cho-Choe-Kim \cite{Cho2004} first stated a blowup criterion:
  \begin{equation*}
    \limsup_{T\rightarrow T^{*}}(\norm{\rho}_{H^1\cap W^{1,q}}+\norm{u}_{D_0^1})=\infty,
  \end{equation*}
  where $q\in (3,6]$, $D_0^1$ will be defined in \eqref{Notation-Sobolev}, and $T^*$ is  the maximal time of
existence of a strong (or classical) solution.
\item  Huang-Li-Wang \cite{Huang-Li-Wang2013} established  the following Serrin-type blowup criterion:
\begin{equation}\label{blowup-Huang-Li-Wang}
  \limsup_{T\rightarrow T^{*}}(\norm{\div u}_{L^{\infty}(0,T;L^{\infty})}+\norm{u}_{L^s(0,T;L^r)})=\infty,
\end{equation}
for any $r\in(3,\infty]$ and $s\in[2,\infty]$ satisfying $\frac{2}{s}+\frac{3}{r}\leq 1$. It is easy to check that $\norm{\div u}_{L^{\infty}(0,T;L^{\infty})}$ in \eqref{blowup-Huang-Li-Wang} can be replaced by $\norm{\rho}_{L^{\infty}(0,T;L^{\infty})}$ (see for instance \cite{Huang-Li2013}).
The similar blowup criterion also holds for the barotropic viscous compressible fluids \cite{Huang2011-2}, which is later extended by Wang \cite{Wang2020} to a more general weak $L^r$-space case as following:
\begin{equation}\label{blowup-Wang2020}
  \limsup_{T\rightarrow T^{*}}(\norm{\rho}_{L^{\infty}(0,T;L^{\alpha})}+\norm{\sqrt{\rho}u}_{L^s(0,T;L^r_w)})=\infty,
\end{equation}
for any $\alpha$, $s$ and $r$ satisfying $\alpha\geq\alpha_0$ and
\begin{equation}\label{index-blowup}
  \begin{cases}
    \frac{2}{s}+\frac{3}{r}\leq 1, & \textrm{if $r=\infty$, or $\alpha_0=\infty$ and $3<r<\infty$},  \\
    \frac{2}{s}+\frac{3}{r}<1, & \textrm{if $\alpha_0<\infty$ and $3<r<\infty$}.
  \end{cases}
\end{equation}
  \item Fan-Jiang-Ou \cite{Fan-Jiang2010} gave the following blowup criterion:
      \begin{equation}\label{blowup-Fan-Jiang}
        \limsup_{T\rightarrow T^{*}}(\norm{\theta}_{L^{\infty}(0,T;L^{\infty})}+\norm{\nabla u}_{L^1(0,T;L^{\infty})})=\infty
      \end{equation}
      under a stringent condition that
      \begin{equation}\label{restrictive-condition1}
        7\mu>\lambda.
      \end{equation}
      Later, Huang-Li \cite{Huang-Li2009} and Huang-Li-Xin \cite{Huang2011-1} succeeded in removing the restriction \eqref{restrictive-condition1} to establish the following blowup criterion:
      \begin{equation*}
        \limsup_{T\rightarrow T^{*}}(\norm{\theta}_{L^{\infty}(0,T;L^{\infty})}+\norm{\mathcal{D} (u)}_{L^1(0,T;L^{\infty})})=\infty.
      \end{equation*}
  \item In the absence of a vacuum, Sun-Wang-Zhang \cite{Sun-Wang-Zhang2011} obtained a blowup criterion in terms of the density and temperature for the initial-boundary value problem:
      \begin{equation*}
        \limsup_{T\rightarrow T^{*}}(\norm{\theta}_{L^{\infty}(0,T;L^{\infty})}+\norm{(\rho,\rho^{-1})}_{L^{\infty}(0,T;L^{\infty})})=\infty
      \end{equation*}
      under the restrictive condition \eqref{restrictive-condition1}.
      \item   Wen-Zhu \cite{Wen2013} removed the restriction $\norm{\rho^{-1}}_{L^{\infty}(0,T;L^{\infty})}$ for the Cauchy problem with initial vacuum and vanishing far field conditions $\tilde{\rho}=\tilde{\theta}=0$ to get the following blowup criterion:
\begin{equation}\label{blowup-Nash}
  \limsup_{T\rightarrow T^{*}}(\norm{\rho}_{L^{\infty}(0,T;L^{\infty})}+\norm{\theta}_{L^{\infty}(0,T;L^{\infty})})=\infty,
\end{equation}
under a more stringent condition that
\begin{equation}\label{restrictive-condition2}
        3\mu>\lambda.
      \end{equation}
      \item Recently, Feireisl-Wen-Zhu \cite{Feireisl2024} succeeded in removing the restriction \eqref{restrictive-condition1} to obtain the following blowup criterion for the Cauchy problem and Dirichlet boundary value problem:
\begin{equation}\label{blowup-Feireisl}
  \limsup_{T\rightarrow T^{*}}(\norm{\rho}_{L^{\infty}(0,T;L^{\infty})}+\norm{\theta-\tilde{\theta}}_{L^s(0,T;L^r)})=\infty,
\end{equation}
for any $r\in(\frac{3}{2},\infty]$ and $s\in[1,\infty]$ satisfying $\frac{2}{s}+\frac{3}{r}\leq 2$.
This criterion can be regarded as a rigorous justification of Nash's conjecture \cite{Nash1958}:
possible singularities must first appear at the level of thermodynamic variables--the density and the temperature--and not for the fluid velocity  in the equations of fluid dynamics.
\end{itemize}

\bigskip
The main purpose of this paper is to improve all the previous blowup criterion results for the
full compressible Navier-Stokes equations \eqref{FCNS-eq} by removing the stringent condition
\eqref{restrictive-condition1}, and allowing initial vacuum states, and furthermore, instead of \eqref{blowup-Huang-Li-Wang} and \eqref{blowup-Feireisl},
 by describing
the blowup mechanism only in terms of a weak Serrin-type criterion.\par
Before stating our main results, let us introduce the following notations and conventions used throughout this paper. We denote
\begin{equation*}
  \int f=\int_{\Omega}fdx.
\end{equation*}
For $1\leq r\leq\infty$ and integer $k\geq1$, we denote the standard Sobolev spaces as follows:
\begin{equation}\label{Notation-Sobolev}
  \begin{cases}
     L^r=L^r(\Omega),\,D^{k,r}=\{u\in L_{loc}^1(\Omega): \norm{\nabla^ku}_{L^r}<\infty\},\\
     W^{k,r}=L^r\cap D^{k,r},\,H^k=W^{k,2},\, D^k=D^{k,2},\\
     D_0^1=\{u\in L^6: \norm{\nabla u}_{L^2}<\infty,\textrm{ and \eqref{Cauchy-condition} or \eqref{Dirichlet-condition} or \eqref{Navier-slip-condition} holds}\},\\
    H_0^1=L^2\cap D_0^1,\,\norm{u}_{D^{k,r}}=\norm{\nabla^ku}_{L^r}.
  \end{cases}
\end{equation}

Next, let us give the definition of strong solutions to the system \eqref{FCNS-eq} in $\Omega\times(0,T)$.

\begin{defi}[Strong solutions]\label{defi-strong-solution}
  $(\rho,u,\theta)$ is called a strong solution to the system \eqref{FCNS-eq} in $\Omega\times(0,T)$, if $(\rho,u, \theta)$ satisfies \eqref{FCNS-eq} a.e. in $\Omega\times(0,T)$, and for some $q\in(3,6]$,
  \begin{equation}\label{strong-sol-defi}
    \begin{aligned}
    &\rho\geq0, \rho-\tilde{\rho}\in C([0,T], W^{1,q}\cap H^1),\, \rho_t\in C([0,T], L^{q}\cap L^2),\\
    &(u,\theta-\tilde{\theta})\in C([0,T], D_0^1\cap D^2)\cap L^2(0,T; D^{2,q}),\\
    &(\sqrt{\rho}u_t,\sqrt{\rho}\theta_t)\in L^{\infty}(0,T; L^2),\,(u_t,\theta_t)\in L^2(0,T; D_0^1).
    \end{aligned}
  \end{equation}
\end{defi}

\noindent\textbf{Initial data.} To match the regularity class specified in Definition \ref{defi-strong-solution}, we assume that the initial data satisfy $\rho_0\geq0$, $\rho_0-\tilde{\rho}\in W^{1,q}\cap H^1$ for some $q\in(3,6]$ and $(u_0,\theta_0-\tilde{\theta})\in D_0^1\cap D^2$. In addition, we suppose that $\rho_0|u_0|^2+\rho_0|\theta_0-\tilde{\theta}|^2\in L^1$, and the following compatibility conditions:
\begin{equation}\label{compatibility-condition}
  \begin{cases}
    \mu\Delta u_0+(\mu+\lambda)\nabla\div u_0-\nabla P(\rho_0,\theta_0)=\sqrt{\rho_0}g_1, \\
    \kappa\Delta\theta_0+2\mu|\mathcal{D}(u_0)|^2+\lambda(\div u_0)^2=\sqrt{\rho_0}g_2,
  \end{cases}
\end{equation}
for some $g_1,\,g_2\in L^2$. Finally, we require $(\rho_0,u_0,\theta_0)$ to satisfy three types of boundary conditions as \eqref{Cauchy-condition}, \eqref{Dirichlet-condition} or \eqref{Navier-slip-condition}.

\begin{rem}\label{rem-local-existence}
Under the stated assumptions on the initial data above, Cho-Choe-Kim \cite{Cho2006-1} proved the local existence of strong solutions except for the boundary condition $\frac{\pa\theta}{\pa n}|_{\pa\Omega}=0$ and the Navier-slip boundary condition \eqref{Navier-slip-condition}. However, we can expect that the local existence in such two cases could be established in a way similar to \cite{Cho2006-1}, especially for the Navier-slip boundary condition, by introducing the Lam\'{e} operator and its spectrum in \cite{Hoff2012}. In particular, the strong solution always exists on a small time interval for the initial data belonging to the class specified above. Furthermore, the life span of the local strong solutions can be always extended beyond the existing one as long as uniform bounds are obtained. Thus any strong solution is defined upon a maximal existence time $T^*>0$.
\end{rem}

Now, we are in a position to state our main results, which are given in the following three theorems.
\begin{thm}\label{thm-blowup-FCNS}
  Let $(\rho,u,\theta)$ be a strong solution to the initial-boundary value problem \eqref{FCNS-eq}\eqref{initial-data} with \eqref{Cauchy-condition} or \eqref{Dirichlet-condition} or \eqref{Navier-slip-condition} satisfying \eqref{strong-sol-defi}. If $T^*<\infty$ is the maximal time of existence, then
  \begin{equation}\label{blowup-goal1}
  \limsup_{T\rightarrow T^{*}}(\norm{\rho}_{L^{\infty}(0,T;L^{\infty})}+\norm{\theta-\tilde{\theta}}_{L^s(0,T;L^r_w)})=\infty,
\end{equation}
for any $s\in[1,\infty]$ and $r\in(\frac{3}{2},\infty]$ satisfying
\begin{equation}\label{index-blowup1}
  \frac{2}{s}+\frac{3}{r}\leq 2,
\end{equation}
where $L^r_w$ denotes the weak-$L^r$ space.
\end{thm}

\begin{rem}
  Compared  to \eqref{blowup-Feireisl} in Feireisl-Wen-Zhu \cite{Feireisl2024}, the main novelty lies in that
  the Serrin norm of the absolute
temperature  is relaxed to weak Serrin norm, and the Navier-slip boundary condition is also considered in our case.
\end{rem}

\begin{thm}\label{thm-blowup-FCNS-velocity}
Let $(\rho,u,\theta)$ be a strong solution to the initial-boundary value problem \eqref{FCNS-eq}\eqref{initial-data} with \eqref{Cauchy-condition} or \eqref{Dirichlet-condition} or \eqref{Navier-slip-condition} satisfying \eqref{strong-sol-defi}.
If $T^*<\infty$ is the maximal time of existence, then
 \begin{equation}\label{blowup-goal2}
  \limsup_{T\rightarrow T^{*}}(\norm{\rho}_{L^{\infty}(0,T;L^{\infty})}+\norm{u}_{L^s(0,T;L^r_w)})=\infty,
\end{equation}
for any  $s\in[2,\infty]$ and $r\in(3,\infty]$ satisfying
 \begin{equation}\label{index-blowup2}
  \frac{2}{s}+\frac{3}{r}\leq 1.
\end{equation}
\end{thm}

\begin{rem}
  Compared to Huang-Li-Wang  \cite{Huang-Li-Wang2013} and \cite{Huang-Li2013}, the main novelty can be outlined as follows.
  First, the Serrin norm of the velocity is relaxed to weak Serrin norm.
  Second, we succeed in removing  the technical restriction $\tilde{\theta}=0$, which plays a crucial role in controlling  $\displaystyle\int\rho|\theta|^2$  in \cite{Huang-Li-Wang2013}
  and \cite{Huang-Li2013}.
  Third, the Navier-slip boundary condition is also considered in our case.
\end{rem}

Finally, as a by-product, we will give a blowup criterion on the isentropic compressible Navier-Stokes equations,
\begin{equation}\label{CNS-eq}
  \begin{cases}
    \pa_t\rho+\div(\rho u)=0, \\
    \pa_t(\rho u)+\div(\rho u\otimes u)-\mu\Delta u-(\mu+\lambda)\nabla\div u+\nabla P(\rho)=0,\\
  \end{cases}
\end{equation}
by removing the crucial assumption $\rho_0\in L^1$ of Wang's work \cite{Wang2020}. Here the pressure $P=a\rho^{\gamma}$ with $\gamma>1$ and $a>0$.

\begin{thm}\label{thm-blowup-CNS}
 Let $(\rho,u)$ be a strong solution to the Cauchy problem \eqref{CNS-eq}\eqref{initial-data} with \eqref{Cauchy-condition} satisfying \eqref{strong-sol-defi} with $\theta=\tilde{\theta}=0$. If $T^*<\infty$ is the maximal time of existence, then
 \begin{equation}\label{blowup-goal3}
  \limsup_{T\rightarrow T^{*}}(\norm{\rho}_{L^{\infty}(0,T;L^{\alpha})}+\norm{u}_{L^s(0,T;L^r_w)})=\infty,
\end{equation}
for any $\alpha$, $s$ and $r$ satisfying $\alpha\geq\alpha_0$ (sufficiently large) and
\begin{equation}\label{index-blowup3}
  \begin{cases}
    \frac{2}{s}+\frac{3}{r}\leq 1, & \textrm{if $\alpha=\infty$ and $3<r\leq\infty$},  \\
    \frac{2}{s}+\frac{3}{r}<1, & \textrm{if $\alpha<\infty$ and $3<r\leq\infty$}.
  \end{cases}
\end{equation}
In particular, if $\alpha<\infty$, it should hold that $\tilde{\rho}=0$.
\end{thm}

\begin{rem}
  In fact, we can  replace $u$ by $\rho^{\min\{\gamma-1,\frac{1}{2}\}}u$ in \eqref{blowup-goal3} to get the following blowup criterion
  \begin{equation}\label{blowup-goal-rem}
  \limsup_{T\rightarrow T^{*}}(\norm{\rho}_{L^{\infty}(0,T;L^{\alpha})}+\norm{\rho^{\min\{\gamma-1,\frac{1}{2}\}}u}_{L^s(0,T;L^r_w)})=\infty.
\end{equation}
This result is slightly different from Wang's work in \cite{Wang2020}. However, we succeed in remove the technical conditions $\rho_0\in L^1$ and $\rho_0|u_0|^2\in L^1$. We also point out that the index condition \eqref{index-blowup3} is a little different from \eqref{index-blowup} for the case $r=\infty$, $\alpha<\infty$. This difference comes from the part of $L^s(0,T;L^r_w)$-norm.
\end{rem}

Now, let us make some comments on the analysis of this paper. For the proof of Theorem \ref{thm-blowup-FCNS}, we first
derive the good estimate on $L_t^{\infty}L_x^2$-norm of $\nabla u$. Compared to \cite{Feireisl2024}, we need to develop some new thoughts to
deal with the difficulty arising from the weak Serrin norm.
To overcome this difficulty, we will make full use of good properties of Lorentz spaces, and interpolation tricks to obtain the desired  estimate on $\norm{\nabla u}_{L_t^{\infty}L_x^2}$ under the conditions of Theorem \ref{thm-blowup-FCNS}. The main observation in this process is that the temperature equation $\eqref{FCNS-eq}_3$ can provide a good control on the diffusion term $\displaystyle\int|u|^2|\div u|^2$. In addition, compared to \cite{Feireisl2024}, we also need to tackle with the trouble boundary terms coming from the Navier-slip boundary case carefully.
 With the key estimate of $\norm{\nabla u}_{L_t^{\infty}L_x^2}$ in hand, we can adjust the methods of \cite{Feireisl2024}
 to prove \eqref{blowup-goal1}, and thus complete the proof of Theorem \ref{thm-blowup-FCNS}. For the proof of Theorem \ref{thm-blowup-FCNS-velocity}, due to the absence of $\tilde{\theta}=0$, we cannot follow the methods of \cite{Huang-Li-Wang2013} and \cite{Huang-Li2013} directly. Indeed, the technical restriction $\tilde{\theta}=0$ plays a key role in controlling  $\displaystyle\int\rho|\theta|^2$  in \cite{Huang-Li-Wang2013}
  and \cite{Huang-Li2013}.
  However, since $\tilde{\theta}$ may be postive in our case,
  we cannot get $\theta-\tilde{\theta}\geq0$ from the parabolic theory of temperature equation $\eqref{FCNS-eq}_3$ and then derive a good bound on the temperature as in \cite{Huang-Li-Wang2013} and \cite{Huang-Li2013}.
   To overcome this difficulty, we will take full advantage of the temperature equation. To see this, by multiplying the equation by $\theta-\tilde{\theta}$ and applying some new estimates, we succeed in getting the key bound on $\displaystyle\int\rho|\theta-\tilde{\theta}|^2$ (see \eqref{theta-L^2-estimate-blowup-u1} for details). Then, combining this key bound with the estimate on $\displaystyle\int|\nabla u|^2$, we finally obtain the bound of $\norm{\nabla u}_{L_t^{\infty}L_x^2}$.
   Once the bound of $\norm{\nabla u}_{L_t^{\infty}L_x^2}$ is obtained, we can employ similar methods of \cite{Huang-Li-Wang2013} and interpolation tricks in Lorentz spaces
  to prove Theorem \ref{thm-blowup-FCNS-velocity}. For the proof of Theorem \ref{thm-blowup-CNS} concerning the isentropic compressible Navier-Stokes system, the key point lies in how to remove the technical assumption $\rho_0\in L^1$ in \cite{Wang2020}. The methods in \cite{Wang2020} are not applicable at the first step of estimating $\norm{\nabla u}_{L_t^{\infty}L_x^2}$.
   To overcome this difficulty, the main observation here is that by combining the estimate of $\displaystyle\int|\nabla u|^2$ with the estimate $\displaystyle\int\rho^2$ in a clever way, we can succeed in obtaining the key bound of $\norm{\nabla u}_{L_t^{\infty}L_x^2}$ by removing the technical assumption $\rho_0\in L^1$. We refer to the proof of \eqref{u-1rd-CNS-u} for details.

The rest of the paper is organized as follows. In Section 2, we introduce some elementary lemmas which will be needed later. In Sections 3-5, we will prove Theorem \ref{thm-blowup-FCNS}, \ref{thm-blowup-FCNS-velocity} and \ref{thm-blowup-CNS}, respectively.

\section{Preliminary}
This section mainly introduces some elementary lemmas used later. First, we notice that
\begin{equation*}
  \Delta u=\nabla\div u-\nabla\times\curl u,
\end{equation*}
which together with the singular integral and  standard elliptic estimates implies that
\begin{lem}
    (1) Let $\Omega=\mathbb{R}^3$ and any $p\in(1,\infty)$. Then it holds that for any integer $k\geq 0$,
    \begin{equation*}
    \norm{u}_{W^{k+1,p}}\lesssim\norm{\div u}_{W^{k,p}}+\norm{\curl u}_{W^{k,p}}.
  \end{equation*}
  In particular,
  \begin{equation}\label{Hodge-decomposition1}
      \norm{\nabla u}_{L^p}\lesssim\norm{\div u}_{L^p}+\norm{\curl u}_{L^p}.
  \end{equation}
  Here and after $a\lesssim b$ means that $a\leq Cb$ for some generic constant $C>0$.

    (2) Let $k\geq 0$ be an integer and $p\in(1,\infty)$, and assume $\Omega$ is a simply connected bounded domain in $\mathbb{R}^3$ with $C^{k+1,1}$ boundary $\pa\Omega$. If $u\cdot n=0$ or $u\times n=0$ on $\pa\Omega$, there exists a positive constant $C(q,k,\Omega)$ such that
  \begin{equation*}
    \norm{u}_{W^{k+1,p}}\leq C(\norm{\div u}_{W^{k,p}}+\norm{\curl u}_{W^{k,p}}).
  \end{equation*}
  In particular, for $k=0$, we have
  \begin{equation}\label{Hodge-decomposition2}
    \norm{\nabla u}_{L^p}\leq C(\norm{\div u}_{L^p}+\norm{\curl u}_{L^p}).
  \end{equation}
\end{lem}
\begin{proof}The first part is a consequence of the singular integral, while the second part comes from \cite{Aramaki2014} and \cite{von-Wahl1992}.
\end{proof}
Next, we introduce some notations as follows:
\begin{equation*}
  G:=(2\mu+\lambda)\div u-P+P(\tilde{\rho},\tilde{\theta}), \quad \omega:=\nabla\times u, \quad \dot{f}:=f_t+u\cdot\nabla f,
\end{equation*}
which represent the effective viscous flux, vorticity and material derivative of $f$, respectively.
From $\eqref{FCNS-eq}_2$, it is easy to check that
\begin{equation}\label{GW-equation}
  \Delta G=\div(\rho\dot{u}),\quad \mu\Delta\omega=\nabla\times(\rho\dot{u}).
\end{equation}
Then, the singular integral and classical elliptic estimates give the following lemma.

\begin{lem}
  (1) Let $\Omega=\mathbb{R}^3$ and any $p\in(1,\infty)$. Then it holds that
  \begin{equation}\label{GW-estimate1}
    \norm{\nabla G}_{L^p}+\norm{\nabla\omega}_{L^p}\lesssim\norm{\rho\dot{u}}_{L^p}.
  \end{equation}
  (2) Let $\Omega$ be a simply connected bounded smooth domain with boundary condition \eqref{Navier-slip-condition} and $p\in(1,\infty)$. Then it holds that
  \begin{equation}\label{GW-estimate2}
    \norm{\nabla G}_{L^p}+\norm{\nabla\omega}_{L^p}\lesssim\norm{\rho\dot{u}}_{L^p}.
  \end{equation}
\end{lem}
\begin{proof}
  The first part is a direct consequence of the singular integral. Here, we only prove the second part. Since $G$ satisfies
  \begin{equation*}
    \begin{cases}
      \Delta G=\div(\rho\dot{u}),\\
      \nabla G\cdot n=\rho\dot{u}\cdot n \quad\mbox{on }\pa\Omega,
    \end{cases}
  \end{equation*}
  where $\nabla\times\omega\cdot n=0$ on $\pa\Omega$ from \cite{Bendali1985} with the unit outer normal $n$ on $\pa\Omega$, one gets that for any integer $k\geq 0$ and $p\in(1,\infty)$,
  \begin{equation*}
    \norm{\nabla G}_{W^{k,p}}\lesssim\norm{\rho\dot{u}}_{W^{k,p}}.
  \end{equation*}
  Due to $\omega\times n=0$ on $\pa\Omega$ and $\div\omega=0$, then  we have from \eqref{Hodge-decomposition2} that
  \begin{equation*}
    \norm{\nabla\omega}_{L^p}\lesssim\norm{\nabla\times\omega}_{L^p}
    \lesssim\norm{\nabla G-\rho\dot{u}}_{L^p}\lesssim\norm{\rho\dot{u}}_{L^p}.
  \end{equation*}
\end{proof}

Next, we introduce the Lorentz spaces $L^{p,q}$ with $1\leq p<\infty, 1\leq q\leq \infty$ and its norm $\norm{\cdot}_{L^{p,q}}$. We say that $f\in L^{p,q}$ if $f\in L_{loc}^1$ and
\begin{equation*}
  \norm{f}_{L^{p,q}}=\begin{cases}
                       \left(\displaystyle\int_0^{\infty}pt^q|\{x\in\Omega: |f(x)|>t\}|^{q/p}\frac{dt}{t}\right)^{1/q},\,q<\infty \\
                       \displaystyle\sup_{t>0}t|\{x\in\Omega: |f(x)|>t\}|^{1/p},\,q=\infty,
                     \end{cases}
\end{equation*}
is bounded. The following lemma concerns some properties of Lorentz spaces.
\begin{lem}
  (1) Let $1\leq p_1,p_2<\infty$, $1\leq q_1, q_2\leq \infty$, $\frac{1}{p}=\frac{1}{p_1}+\frac{1}{p_2}<1$, and $q=\min\{q_1,q_2\}$. Then the following H\"{o}lder inequality holds (see \cite{O'Neil1963}):
  \begin{equation}\label{Lorentz-Holder}
    \norm{fg}_{L^{p,q}}\lesssim\norm{f}_{L^{p_1,q_1}}\norm{g}_{L^{p_2,q_2}.}
  \end{equation}
  Indeed, here $q$ satisfies $\frac{1}{q}=\frac{1}{q_1}+\frac{1}{q_2}$. With the increasing of $L^{p,q}$ as $q$ increases, we can replace $q$ by $\min\{q_1,q_2\}$. \\
  (2) Let $1<p_1,p_2<\infty$, $1\leq q_1, q_2\leq \infty$, $0<\theta<1$. Then the interpolation characteristic of Lorentz spaces holds for $p_1\neq p_2$ (see \cite{Chamorro2013}):
  \begin{equation}\label{Lorentz-interpolation}
    (L^{p_1,q_1},L^{p_2,q_2})_{\theta,q}=L^{p,q}
  \end{equation}
  with $\frac{1}{p}=\frac{1-\theta}{p_1}+\frac{\theta}{p_2}$. When $p_1=p_2$, this interpolation also holds with $\frac{1}{q}=\frac{1-\theta}{q_1}+\frac{\theta}{q_2}$.\\
  (3) The Lorentz spaces $L^{p,q}$ increase as the exponent $q$ increases (see \cite{Grafakos2014}):
  \begin{equation}\label{Lorentz-increase}
    \norm{f}_{L^{p,q_2}}\leq (\frac{q_1}{p})^{\frac{1}{q_1}-\frac{1}{q_2}}\norm{f}_{L^{p,q_1}}
  \end{equation}
  with $1\leq p<\infty$ and $1\leq q_1<q_2\leq\infty$. In particular, $L^p\subseteq L^{p,\infty}=L^p_w$.
\end{lem}

\begin{proof}
  Here, we introduce an equivalent definition to understand the H\"{o}lder inequality readily. Let $f_m=f\chi_{\{x\in\Omega: 2^m\leq|f|<2^{m+1}\}}$ with $\chi_E$ being a characteristic function on some disjoint sets $E_m=\{x\in\Omega: 2^m\leq|f|<2^{m+1}\}$. Then
  \begin{equation*}
    \begin{aligned}
    \norm{f}_{L^{p,q}}^q&:=p\int_0^{\infty}t^q|\{x\in\Omega: |f(x)|>t\}|^{\frac{q}{p}}\frac{dt}{t}\\
    &\sim\sum_m\int_{2^m}^{2^{m+1}}(2^m\sum_{n\geq m}|E_n|^{\frac{1}{p}})^q\frac{dt}{t}\\
    &\sim\norm{2^m\sum_{n\geq m}|E_n|^{\frac{1}{p}}}_{l_m^{q}}^q.
    \end{aligned}
  \end{equation*}
  Thus
  \begin{equation*}
    \norm{2^m|E_m|^{\frac{1}{p}}}_{l_m^q}\lesssim\norm{f}_{L^{p,q}}\lesssim\norm{\sum_{k\geq 0}2^{-k}2^{m+k}|E_{m+k}|^{\frac{1}{p}}}_{l_m^{q}}\lesssim\norm{2^m|E_m|^{\frac{1}{p}}}_{l_m^q},
  \end{equation*}
  where we have used Young's inequality in the last inequality. Therefore, we have the following equivalence:
  \begin{equation}\label{Lorentz-equivalence}
    \norm{f}_{L^{p,q}}\sim\norm{\norm{f_m}_{L^p(\Omega)}}_{l_m^q(\mathbb{Z})}.
  \end{equation}
  Then, H\"{o}lder inequality and interpolation inequality just come from the cases in $L^q(\Omega)$ and $l^q(\mathbb{Z})$. Also the monotonicity of $L^{p,q}$ on $q$ follows the case of $l^q(\mathbb{Z})$.
\end{proof}

 The next lemma is introduced from \cite{Feireisl2004B} by Feireisl. Here, we give a direct proof which depends on Poincar\'{e} inequality instead of the weak-$L^p$ convergence.
\begin{lem}\label{lemma-Feireisl}
  (see \cite{Feireisl2004B}) Let $v\in W^{1,2}(\Omega)$, and $\rho$ be a non-negative function such that
  \begin{equation*}
    0<M\leq\int_{\Omega}\rho dx,\quad \int_{\Omega}\rho^{\gamma}\leq E_0,
  \end{equation*}
  where $\Omega\subset\mathbb{R}^N$ is a bounded domain for $N\geq 1$ and $\gamma>1$.
Then there exists a constant $c>0$ depending on $M,E_0,\Omega,\gamma$ such that
  \begin{equation*}
    \norm{v}_{L^2(\Omega)}^2\leq c(E_0,M,\Omega,\gamma)\left[\norm{\nabla v}_{L^2(\Omega)}^2+\left(\int_{\Omega}\rho|v|dx\right)^2\right].
  \end{equation*}
\end{lem}
\begin{proof}
  Divide $\norm{v}_{L^2}^2$ into two parts as follows,
  \begin{equation*}
    \norm{v}_{L^2}^2=\int_{|v|\leq k}|v|^2dx+\int_{|v|>k}|v|^2dx.
  \end{equation*}
  Then, by denoting $v_k=v1_{|v|\leq k}$ with $1_{|v|\leq k}$ being the characteristic function on $\{|v|\leq k\}$, we have
  \begin{equation*}
    \norm{v_k}_{L^2(\Omega)}\leq \norm{\tilde{v}_k}_{L^2(\Omega)}+C(\Omega)|\bar{v}_k|\leq C(\Omega)(\norm{\nabla v_k}_{L^2(\Omega)}+|\bar{v}_k|)\leq C(\Omega)(\norm{\nabla v}_{L^2(\Omega)}+|\bar{v}_k|),
  \end{equation*}
  where $\bar{f}$ denotes the integral average of $f$ on $\Omega$ and $\tilde{f}=f-\bar{f}$.
  It is clear that
  \begin{equation*}
    \bar{M}\bar{v}_k=\int_{\Omega}\rho(v_k-\tilde{v}_k)dx,
  \end{equation*}
 where $\bar{M}:=\int_{\Omega}\rho dx\geq M$. Then, by H\"{o}lder inequality and Poincar\'{e} inequality, we have
  \begin{equation*}
  \begin{aligned}
    \left|\displaystyle\int_{\Omega}\rho\tilde{v}_kdx\right|&\leq\norm{\rho}_{L^{\gamma}}\norm{\tilde{v}_k}_{L^{\gamma'}}\leq C(\gamma,\Omega)\norm{\rho}_{L^{\gamma}}\norm{\tilde{v}_k}_{L^{2\gamma'}}\\
    &\leq C(\gamma,\Omega)\norm{\rho}_{L^{\gamma}}k^{1-\frac{1}{\gamma'}}\norm{\tilde{v}_k}_{L^2}^{\frac{1}{\gamma'}}\\
    &\leq C(\gamma,\Omega)\norm{\rho}_{L^{\gamma}}k^{1-\frac{1}{\gamma'}}\norm{\nabla v}_{L^2}^{\frac{1}{\gamma'}}
  \end{aligned}
  \end{equation*}
  with $\frac{1}{\gamma}+\frac{1}{\gamma'}=1$ and the fact that $|\tilde{v}_k|\leq 2k$.
  Thus, we have
  \begin{equation*}
    |\bar{v}_k|\leq C(M,E_0,\Omega,\gamma)\left(\int_{\Omega}\rho|v|dx+k^{1-\frac{1}{\gamma'}}\norm{\nabla v}_{L^2}^{\frac{1}{\gamma'}}\right).
  \end{equation*}
  On the other hand, by Gagliardo-Nirenberg interpolation inequality, we have
  \begin{equation*}
    \int_{|v|>k}|v|^2dx\leq\begin{cases}
                             k^{-2}\norm{v}_{L^4}^4\leq C(\Omega)k^{-2}(\norm{v}_{L^2}^2\norm{\nabla v}_{L^2}^2+\norm{v}_{L^2}^4), & \mbox{if } N=1,2, \\
                             k^{-\frac{2}{N-1}}\norm{v}_{L^{\frac{2N}{N-1}}}^{\frac{2N}{N-1}}
                             \leq C(\Omega)k^{-\frac{2}{N-1}}(\norm{v}_{L^2}^{\frac{N}{N-1}}\norm{\nabla v}_{L^2}^{\frac{N}{N-1}}+\norm{v}_{L^2}^{\frac{2N}{N-1}}), & \mbox{if }N\geq 3.
                           \end{cases}
  \end{equation*}
  By choosing $k=C_1\norm{v}_{L^2}>0$ with $C_1>0$ large enough (since the case $\norm
  {v}_{L^2}=0$ is trivial) and using Young's inequality, we get
  \begin{equation*}
    \int_{|v|>k}|v|^2dx\leq\frac{1}{4}\norm{v}_{L^2}^2+C(\Omega)\norm{\nabla v}_{L^2}^2.
  \end{equation*}
  Therefore, combining all the above estimates, we have from Young's inequality that
  \begin{equation*}
  \begin{aligned}
    \norm{v}_{L^2}^2&\leq \frac{1}{4}\norm{v}_{L^2}^2+C(\Omega)\norm{\nabla v}_{L^2}^2+C(M,E_0,\Omega,\gamma)\left[\left(\int_{\Omega}\rho|v|dx\right)^2+\norm{
    v}_{L^2}^{2-\frac{2}{\gamma'}}\norm{\nabla v}_{L^2}^{\frac{2}{\gamma'}}\right]\\
    &\leq\frac{1}{2}\norm{v}_{L^2}^2+C(M,E_0,\Omega,\gamma)\left[\norm{\nabla v}_{L^2(\Omega)}^2+\left(\int_{\Omega}\rho|v|dx\right)^2\right],
  \end{aligned}
  \end{equation*}
  which implies the desired inequality in Lemma \ref{lemma-Feireisl} immediately.

\end{proof}

\begin{rem}
  The constant $c$ in the original lemma of \cite{Feireisl2004B} depends only on $M,E_0$. However, it is unreasonable without considering the bound of $\Omega$. Here, we give a counterexample as follows:

  Let $\Omega=B(0,R+2)\in\mathbb{R}^3$ be a ball originated at zero point with radius $R+2$ for $R>1$ large enough. Then we can define
  \begin{equation*}
    v(x)=f(|x|)=\begin{cases}
                  1, & \mbox{if $|x|\leq R$}, \\
                  0, & \mbox{if $|x|\geq R+1$},
                \end{cases}
    \qquad \rho(x)=g(|x|)=\begin{cases}
                            1, & \mbox{if } |x|\leq 1, \\
                            2-|x|, & \mbox{if } 1<|x|<2, \\
                            0, & \mbox{if } 2\leq|x|\leq R, \\
                            R^{-2}(|x|-R), & \mbox{if } R<|x|<R+1, \\
                            R^{-2}, & \mbox{if }|x|\geq R+1.
                          \end{cases}
  \end{equation*}
  Here, we need $v\in C^1(\Omega)$. Then, it is easy to check that
  \begin{equation*}
    \norm{v}_{L^2}^2\sim R^3+R^2,\quad \norm{\nabla v}_{L^2}^2\lesssim R^2,\quad \int_{\Omega}\rho|v|dx\sim 1,\quad \int_{\Omega}\rho dx\sim 1,\quad \int_{\Omega}\rho^{\gamma}dx\sim 1,
  \end{equation*}
  which implies that the constant $c$ depends on $R>1$, otherwise the inequality in Lemma \ref{lemma-Feireisl} is invalid for large $R>1$.
\end{rem}
\begin{rem}
  For the boundary condition \eqref{Dirichlet-condition} or \eqref{Navier-slip-condition}, the solution in Theorem \ref{thm-blowup-FCNS} satisfies the mass conservation:
  $$\int\rho=\int\rho_0:=M_0>0.$$
  Thus, under the condition of Lemma \ref{lemma-Feireisl}, we have the following estimate by H\"{o}lder inequality:
  \begin{equation}\label{L2-control-lem}
    \norm{v}_{L^2(\Omega)}^2\leq c(E_0,M_0)\left(\norm{\nabla v}_{L^2(\Omega)}^2+\int_{\Omega}\rho|v|^2dx\right).
  \end{equation}
\end{rem}

The following trace inequality comes from Theorem 1.6.6 in \cite{Brenner2008}.
\begin{lem}
Suppose that the domain $\Omega$ has a Lipschitz boundary and $p\in[1,\infty]$. Then, there exists a positive constant $C$ such that
  \begin{equation}\label{trace-inequality}
    \norm{v}_{L^p(\pa\Omega)}\leq C\norm{v}_{L^p(\Omega)}^{1-\frac{1}{p}}\norm{v}_{W^{1,p}(\Omega)}^{\frac{1}{p}}, \quad \forall v\in W^{1,p}(\Omega).
  \end{equation}
\end{lem}

Recall the Lam\'{e}'s system (see \cite{Cai-Li2023,Nirenberg1959,Agmon-Douglis-Nirenberg1964}):
\begin{equation}\label{Lame-system}
  \begin{cases}
    Lu=\mu\Delta u+(\mu+\lambda)\nabla\div u=f,\\
    \mbox{\eqref{Cauchy-condition} or \eqref{Dirichlet-condition} or \eqref{Navier-slip-condition} holds}.
  \end{cases}
\end{equation}
Then, for any $p\in(1,\infty)$ and integer $k\geq 0$, there exists a positive constant $C$ such that
\begin{equation}\label{Lame-estimate}
\begin{aligned}
\norm{u}_{W^{k+2,p}}&\leq C(\norm{f}_{W^{k,p}}+\norm{u}_{L^p}), \\
\norm{u}_{W^{k+1,p}}&\leq C(\norm{g}_{W^{k,p}}+\norm{u}_{L^p}), \mbox{ for }f=\nabla g.
\end{aligned}
\end{equation}
In particular, for the case \eqref{Dirichlet-condition} and \eqref{Navier-slip-condition}, the term $\norm{u}_{L^p}$ on the righthand side of inequalities can be removed. For the case \eqref{Cauchy-condition}, the disappearance of $\norm{u}_{L^p}$ only occurs at the estimate of $\norm{\nabla^{k+2}u}_{L^p}$ in the first inequality or $\norm{\nabla^{k+1}u}_{L^p}$ in the second.

Here, we also introduce the decomposition $u=g+h$ from \cite{Sun2011} to deal with the Dirichlet problem, where $h$ is the unique solution to
\begin{equation}\label{h-pressure-eq}
  \begin{cases}
    Lh=\nabla P, & \mbox{in } \Omega, \\
    h|_{\pa\Omega}=0, & \mbox{if $\Omega$ is bounded},\\
    h\rightarrow 0, & \mbox{as }|x|\rightarrow\infty, \mbox{ if }\Omega=\mathbb{R}^3,
  \end{cases}
\end{equation}
and $g$ satisfies
\begin{equation}\label{g-material-eq}
  \begin{cases}
    Lg=\rho\dot{u}, & \mbox{in } \Omega, \\
    g|_{\pa\Omega}=0, & \mbox{if $\Omega$ is bounded},\\
    g\rightarrow 0, & \mbox{as }|x|\rightarrow\infty, \mbox{ if }\Omega=\mathbb{R}^3.
  \end{cases}
\end{equation}
From \eqref{Lame-estimate}, when $\Omega$ is bounded, we have
\begin{equation}\label{h-g-estimate-Dirichlet}
  \norm{h}_{W^{1,p}}\leq C\norm{P(\rho,\theta)-P(\tilde{\rho},\tilde{\theta})}_{L^p},\quad\norm{\nabla^2h}_{L^p}\leq C\norm{\nabla P}_{L^p},\quad \norm{g}_{W^{2,p}}\leq C\norm{\rho\dot{u}}_{L^p},
\end{equation}
and
\begin{equation}\label{h-g-estimate-Whole}
  \norm{\nabla h}_{L^p}\leq C\norm{P(\rho,\theta)-P(\tilde{\rho},\tilde{\theta})}_{L^p},\quad\norm{\nabla^2h}_{L^p}\leq C\norm{\nabla P}_{L^p},\quad \norm{\nabla^2g}_{L^p}\leq C\norm{\rho\dot{u}}_{L^p},
\end{equation}
for any $p\in(1,\infty)$. Furthermore, there exists an estimate on $h$ in Lorentz spaces as follows:
\begin{equation}\label{h-estimate-Lorentz}
  \norm{\nabla h}_{L^{p,\infty}}\leq C\norm{P(\rho,\theta)-P(\tilde{\rho},\tilde{\theta})}_{L^{p,\infty}},
\end{equation}
for any $p\in(1,\infty)$. One can refer to \cite{Adimurthi2021,Coifman1974} for details.

\section{Proof of Theorem \ref{thm-blowup-FCNS}}\label{sect-proof-theorem-FCNS-theta}
 Assume that $T^*<\infty$ and there exist constants $r\in(\frac{3}{2},\infty]$ and $s\in[1,\infty]$ satisfying \eqref{index-blowup1}, i.e.
 \begin{equation*}
   \frac{2}{s}+\frac{3}{r}\leq 2
 \end{equation*}
  such that
\begin{equation}\label{assumption1}
  \norm{\rho}_{L^{\infty}(0,T;L^{\infty})}+\norm{\theta-\tilde{\theta}}_{L^s(0,T;L^r_w)}\leq M^*<\infty,
\end{equation}
for any $T\in(0,T^*)$. Our aim is to show that under the assumption \eqref{assumption1} and the conditions of Theorem \ref{thm-blowup-FCNS}, there exists a upper bound $C>0$ depending only on $M^*$, $T^*$, initial data and the viscosity coefficients such that
\begin{equation}\label{goal-inequality}
\begin{aligned}
  &\sup_{0\leq t<T^*}(\norm{\rho-\tilde{\rho}}_{W^{1,q}\cap H^1}+\norm{(u,\theta-\tilde{\theta})}_{D_0^1\cap D^2}+\norm{\rho_t}_{L^q\cap L^2}+\norm{(\sqrt{\rho}u_t,\sqrt{\rho}\theta_t)}_{L^2})\\
  &\qquad+\int_0^{T^*}(\norm{(u_t,\theta_t)}_{D^1}^2+\norm{(u,\theta)}_{D^{2,q}}^2)dt\leq C.
\end{aligned}
\end{equation}
It should be noted that when $\Omega$ is bounded, one can take $\tilde{\rho}=\tilde{\theta}=0$ in \eqref{goal-inequality}. For simplicity, we involve the constants $\tilde{\rho}$ and $\tilde{\theta}$ even in the context of the initial-boundary value problem.

\smallskip

In the rest of this section, we devote ourselves to deriving the uniform energy estimate \eqref{goal-inequality}.
\subsection{Estimates on 0-order energy.}
In this subsection, we will give several lemmas on the 0-order estimates which will be used in the next subsection.

\begin{lem}\label{lem-u-L^4}
Under the assumptions of Theorem \ref{thm-blowup-FCNS} and \eqref{assumption1}, it holds that for any small $\epsilon>0$,
\begin{equation}\label{u-L^4-estimate}
\begin{aligned}
  \frac{d}{dt}\int\rho|u|^4+\mu\int|u|^2|\nabla u|^2&\leq C_{\epsilon}\left(1+\norm{\theta-\tilde{\theta}}_{L^{r,\infty}}^s\right)\left(\norm{\sqrt{\rho}|u|^2}_{L^2}^2+\norm{\sqrt{\rho}(\theta-\tilde{\theta})}_{L^2}^2+\norm{\sqrt{\rho}u}_{L^2}^2\right)\\
  &\qquad+\epsilon\norm{\nabla\theta}_{L^2}^2+C\int|u|^2|\div u|^2.
\end{aligned}
\end{equation}
In particular, $\norm{\sqrt{\rho}u}_{L^2}$ in the right-hand side of this and below inequality in this subsection only appears when $\tilde{\theta}>0$.
\end{lem}
\begin{proof}
  Multiplying \eqref{FCNS-eq}$_2$ by $4|u|^2u$, and then integrating the resultant equation over $\Omega$, we have from integration by parts that
  \begin{equation}\label{L^4-estimate-1}
    \begin{aligned}
    &\quad\frac{d}{dt}\int\rho|u|^4+\int4|u|^2\left(\mu|\nabla u|^2+(\lambda+\mu)|\div u|^2+2\mu|\nabla|u||^2\right)\\
    &=4\int P\div(|u|^2u)-8(\lambda+\mu)\int\div u|u|u\cdot\nabla|u|\\
    &\leq C\int\rho(\theta-\tilde{\theta})|u|^2|\nabla u|+C\tilde{\theta}\int\rho|u|^2|\nabla u|+2\mu\int|u|^2|\nabla|u||^2+C\int|u|^2|\div u|^2\\
    &\leq C\int\rho^2|\theta-\tilde{\theta}|^2|u|^2+C\int\rho^2|u|^2+2\mu\int|u|^2|\nabla u|^2+2\mu\int|u|^2|\nabla|u||^2+C\int|u|^2|\div u|^2,
    \end{aligned}
  \end{equation}
  which implies that
  \begin{equation}\label{L^4-estimate-2}
    \begin{aligned}
    &\quad\frac{d}{dt}\int\rho|u|^4+2\mu\int|u|^2\left(\mu|\nabla u|^2+|\nabla|u||^2\right)\\
    &\leq C\int\rho^2|\theta-\tilde{\theta}|^2|u|^2+C\int\rho|u|^2+C\int|u|^2|\div u|^2.
    \end{aligned}
  \end{equation}
   Next, we estimate $\displaystyle I:=C\int\rho^2|\theta-\tilde{\theta}|^2|u|^2$. There exist two cases for the index of blowup criterion \eqref{index-blowup1} as follows:
  \begin{itemize}
    \item If $r=\infty$, then $L^{\infty}_w=L^{\infty}$. Thus, by virtue of \eqref{assumption1}, we have
    \begin{equation}\label{L^4-estimate-3-1}
    \begin{aligned}
      I&\lesssim\norm{\theta-\tilde{\theta}}_{L^{\infty}}\norm{\sqrt{\rho}(\theta-\tilde{\theta})}_{L^2}\norm{\sqrt{\rho}|u|^2}_{L^2}\\
      &\lesssim(1+\norm{\theta-\tilde{\theta}}_{L^{\infty}}^s)\left(\int\rho(\theta-\tilde{\theta})^2+\int\rho|u|^4\right).
    \end{aligned}
    \end{equation}
    \item If $\frac{3}{2}<r<\infty$, by \eqref{assumption1}, \eqref{Lorentz-Holder}, \eqref{Lorentz-interpolation} and interpolation inequality, we get
    \begin{equation}\label{L^4-estimate-3-2}
    \begin{aligned}
      I&\lesssim\norm{\theta-\tilde{\theta}}_{L^{r,\infty}}\norm{\sqrt{\rho}(\theta-\tilde{\theta})}_{L^{\frac{2r}{r-1},2}}\norm{\sqrt{\rho}|u|^2}_{L^{\frac{2r}{r-1},2}}\\
      &\lesssim\norm{\theta-\tilde{\theta}}_{L^{r,\infty}}\norm{\sqrt{\rho}(\theta-\tilde{\theta})}_{L^{\frac{2r_1}{r_1-1}}}^{\frac{1}{2}}\norm{\sqrt{\rho}(\theta-\tilde{\theta})}_{L^{\frac{2r_2}{r_2-1}}}^{\frac{1}{2}}
      \norm{\sqrt{\rho}|u|^2}_{L^{\frac{2r_1}{r_1-1}}}^{\frac{1}{2}}\norm{\sqrt{\rho}|u|^2}_{L^{\frac{2r_2}{r_2-1}}}^{\frac{1}{2}}\\
      &\lesssim\norm{\theta-\tilde{\theta}}_{L^{r,\infty}}\norm{\sqrt{\rho}(\theta-\tilde{\theta})}_{L^{\frac{2r_1}{r_1-1}}}\norm{\sqrt{\rho}(\theta-\tilde{\theta})}_{L^{\frac{2r_2}{r_2-1}}}\\
      &\qquad+\norm{\theta-\tilde{\theta}}_{L^{r,\infty}}\norm{\sqrt{\rho}|u|^2}_{L^{\frac{2r_1}{r_1-1}}}\norm{\sqrt{\rho}|u|^2}_{L^{\frac{2r_2}{r_2-1}}}\\
      &\lesssim\norm{\theta-\tilde{\theta}}_{L^{r,\infty}}\norm{\sqrt{\rho}|u|^2}_{L^2}^{1-\frac{3}{2r_1}}\norm{\sqrt{\rho}|u|^2}_{L^6}^{\frac{3}{2r_1}}
      \norm{\sqrt{\rho}|u|^2}_{L^2}^{1-\frac{3}{2r_2}}\norm{\sqrt{\rho}|u|^2}_{L^6}^{\frac{3}{2r_2}}\\
      &\quad+\norm{\theta-\tilde{\theta}}_{L^{r,\infty}}\norm{\sqrt{\rho}(\theta-\tilde{\theta})}_{L^2}^{1-\frac{3}{2r_1}}\norm{\sqrt{\rho}(\theta-\tilde{\theta})}_{L^6}^{\frac{3}{2r_1}}
      \norm{\sqrt{\rho}(\theta-\tilde{\theta})}_{L^2}^{1-\frac{3}{2r_2}}\norm{\sqrt{\rho}(\theta-\tilde{\theta})}_{L^6}^{\frac{3}{2r_2}}\\
      &\lesssim\norm{\theta-\tilde{\theta}}_{L^{r,\infty}}\norm{\sqrt{\rho}|u|^2}_{L^2}^{2-\frac{3}{r}}\norm{\sqrt{\rho}|u|^2}_{L^6}^{\frac{3}{r}}
      +\norm{\theta-\tilde{\theta}}_{L^{r,\infty}}\norm{\sqrt{\rho}(\theta-\tilde{\theta})}_{L^2}^{2-\frac{3}{r}}\norm{\sqrt{\rho}(\theta-\tilde{\theta})}_{L^6}^{\frac{3}{r}}\\
      &\leq C_{\epsilon}\norm{\theta-\tilde{\theta}}_{L^{r,\infty}}^{\frac{2}{2-\frac{3}{r}}}\left(\norm{\sqrt{\rho}|u|^2}_{L^2}^2+\norm{\sqrt{\rho}(\theta-\tilde{\theta})}_{L^2}^2\right)+\epsilon\norm{|u|^2}_{L^6}^2+\epsilon\norm{\theta-\tilde{\theta}}_{L^6}^2\\
      &\leq C_{\epsilon}\left(1+\norm{\theta-\tilde{\theta}}_{L^{r,\infty}}^s\right)\left(\norm{\sqrt{\rho}|u|^2}_{L^2}^2+\norm{\sqrt{\rho}(\theta-\tilde{\theta})}_{L^2}^2\right)+C\epsilon\left(\norm{u|\nabla u|}_{L^2}^2+\norm{\nabla\theta}_{L^2}^2\right),
    \end{aligned}
    \end{equation}
    where $\frac{3}{2}<r_1<r<r_2<\infty$ satisfy $\frac{2}{r}=\frac{1}{r_1}+\frac{1}{r_2}$, and  we have used the Sobolev inequality and \eqref{L2-control-lem} for the bounded domain case in the last inequality.
  \end{itemize}

  Consequently, substituting \eqref{L^4-estimate-3-1} and \eqref{L^4-estimate-3-2} into \eqref{L^4-estimate-2}, we can get \eqref{u-L^4-estimate}, and thus complete
  the proof of Lemma \ref{lem-u-L^4}.
\end{proof}
Next, we deduce the estimate on the term $\displaystyle\int|u|^2|\div u|^2$, which is stated in the following lemma.
\begin{lem}\label{lem-divu-L^2}
Under the assumptions of Theorem \ref{thm-blowup-FCNS} and \eqref{assumption1}, for any given $\lambda>0$, it holds that
\begin{equation}\label{divu-L^2-estimate}
  \begin{aligned}
  &\quad2\mu\int|\mathcal{D}(u)|^2|u|^2+\frac{\lambda}{2}\int|\div u|^2|u|^2\\
  &\leq\frac{d}{dt}\int c_v\rho|u|^2(\theta-\tilde{\theta})+\epsilon\int\rho|u_t|^2+\epsilon\int|u|^2|\nabla u|^2+\epsilon\int|\nabla\theta|^2\\
  &\qquad+C_{\epsilon}\left(\norm{\theta-\tilde{\theta}}_{L^{r,\infty}}^s+1\right)\left(\norm{\sqrt{\rho}|u|^2}_{L^2}^2
  +\norm{\sqrt{\rho}(\theta-\tilde{\theta})}_{L^2}^2+\norm{\sqrt{\rho}u}_{L^2}^2\right)
  \end{aligned}
\end{equation}
with $\epsilon>0$ arbitrarily small.
\end{lem}

\begin{proof}
  Multiplying $\eqref{FCNS-eq}_3$ by $|u|^2$ and integrating the resultant equation over $\Omega$, we have
  \begin{equation}\label{divu-L2-estimate1}
    \begin{aligned}
    &\quad 2\mu\int|\mathcal{D}(u)|^2|u|^2+\lambda\int|\div u|^2|u|^2\\
    &=\int c_v\rho|u|^2\theta_t+\int c_v\rho|u|^2u\cdot\nabla\theta+\int R\rho\theta\div u|u|^2-\kappa\int\Delta\theta|u|^2\\
    &=\frac{d}{dt}\int c_v\rho|u|^2(\theta-\tilde{\theta})-\int c_v\left((\rho|u|^2)_t+\div(\rho|u|^2u)\right)(\theta-\tilde{\theta})\\
    &\qquad+\int R\rho\theta\div u|u|^2+2\kappa\int\nabla\theta|u|\nabla|u|.
    \end{aligned}
  \end{equation}
  Employing the discussions used in \eqref{L^4-estimate-3-1} and \eqref{L^4-estimate-3-2}, we have
  \begin{equation}\label{divu-L2-estimate2}
    \begin{aligned}
    \int R\rho\theta\div u|u|^2&=\int R\rho(\theta-\tilde{\theta})\div u|u|^2+\tilde{\theta}\int R\rho\div u|u|^2\\
    &\leq\epsilon\int|\div u|^2|u|^2+C_{\epsilon}\int\rho^2|\theta-\tilde{\theta}|^2|u|^2+C_{\epsilon}\int\rho^2|u|^2\\
    &\leq C_{\epsilon}\left(\norm{\theta-\tilde{\theta}}_{L^{r,\infty}}^s+1\right)\left(\norm{\sqrt{\rho}|u|^2}_{L^2}^2+\norm{\sqrt{\rho}(\theta-\tilde{\theta})}_{L^2}^2+\norm{\sqrt{\rho}u}_{L^2}^2\right)\\
    &\qquad+\epsilon\int|\div u|^2|u|^2+\epsilon\int|\nabla |u||^2|u|^2+\epsilon\norm{\nabla\theta}_{L^2}^2,
    \end{aligned}
  \end{equation}
  where $\epsilon>0$ is sufficient small.
  It is easy to check that
  \begin{equation}\label{divu-L2-estimate3}
    2\kappa\int\nabla\theta|u|\nabla|u|\leq\epsilon\int|u|^2|\nabla |u||^2+C_{\epsilon}\int|\nabla\theta|^2.
  \end{equation}
  Finally, due to $\eqref{FCNS-eq}_1$, we have
  \begin{equation*}
    (\rho|u|^2)_t+\div(\rho|u|^2u)=\rho(|u|^2)_t+\rho u\cdot\nabla|u|^2=2\rho\dot{u}\cdot u.
  \end{equation*}
  Then, by virtue of \eqref{L^4-estimate-3-1} and \eqref{L^4-estimate-3-2} in Lemma \ref{lem-u-L^4}, we have
  \begin{equation}\label{divu-L2-estimate4}
    \begin{aligned}
    &\quad\int c_v[(\rho|u|^2)_t+\div(\rho|u|^2u)](\theta-\tilde{\theta})\\
    &\leq\epsilon\int\rho|u_t|^2+\epsilon\int|u|^2|\nabla u|^2+C_{\epsilon}\int\rho|u|^2(\theta-\tilde{\theta})^2\\
    &\leq\epsilon\int\rho|u_t|^2+\epsilon\int|u|^2|\nabla u|^2+\epsilon\int|\nabla \theta|^2\\
    &\qquad+C_{\epsilon}\left(\norm{\theta-\tilde{\theta}}_{L^{r,\infty}}^s+1\right)\left(\norm{\sqrt{\rho}|u|^2}_{L^2}^2
    +\norm{\sqrt{\rho}(\theta-\tilde{\theta})}_{L^2}^2\right).
    \end{aligned}
  \end{equation}
  Plugging \eqref{divu-L2-estimate2}, \eqref{divu-L2-estimate3} and \eqref{divu-L2-estimate4} into \eqref{divu-L2-estimate1}, we get \eqref{divu-L^2-estimate}. Therefore,
  we have completed the proof of Lemma \ref{lem-divu-L^2}.
\end{proof}
\begin{lem}\label{lem-u-L^4-close}
Under the assumptions of Theorem \ref{thm-blowup-FCNS} and \eqref{assumption1}, the following estimates hold depending on the sign of the bulk viscosity coefficient $\lambda$.
\begin{itemize}
  \item If $\lambda>0$, it holds that
  \begin{equation}\label{u-L^4-close1}
  \begin{aligned}
    &\quad\frac{d}{dt}\int[\rho|u|^4-\frac{4Cc_v}{\lambda}\rho|u|^2(\theta-\tilde{\theta})]+\frac{\mu}{2}\int|u|^2|\nabla u|^2\\
    &\leq C_{\epsilon}\left(\norm{\theta-\tilde{\theta}}_{L^{r,\infty}}^s+1\right)\left(\norm{\sqrt{\rho}|u|^2}_{L^2}^2
  +\norm{\sqrt{\rho}(\theta-\tilde{\theta})}_{L^2}^2+\norm{\sqrt{\rho}u}_{L^2}^2\right)\\
  &\qquad+C\epsilon\int\rho|u_t|^2+C\epsilon\norm{\nabla\theta}_{L^2}^2,
  \end{aligned}
  \end{equation}
  for sufficient small $\epsilon>0$. In particular, the term $\displaystyle\int\rho|u_t|^2$ in \eqref{u-L^4-close1} can also be replaced by $\displaystyle\int\rho|\dot{u}|^2$.
  \item If $\lambda\leq 0$, it holds that
  \begin{equation}\label{u-L^4-close2}
  \begin{aligned}
    &\quad\frac{d}{dt}\int\rho|u|^4+2\mu\int|u|^2|\nabla u|^2\\
    &\leq C_{\epsilon}\left(\norm{\theta-\tilde{\theta}}_{L^{r,\infty}}^s+1\right)\left(\norm{\sqrt{\rho}|u|^2}_{L^2}^2
  +\norm{\sqrt{\rho}(\theta-\tilde{\theta})}_{L^2}^2+\norm{\sqrt{\rho}u}_{L^2}^2\right)+\epsilon\norm{\nabla\theta}_{L^2}^2,
  \end{aligned}
  \end{equation}
  for sufficient small $\epsilon>0$. In particular, the term $\displaystyle\int\rho|u_t|^2$ in \eqref{u-L^4-close2} can also be replaced by $\displaystyle\int\rho|\dot{u}|^2$.
\end{itemize}
\end{lem}

\begin{proof}
  Combining \eqref{u-L^4-estimate} with \eqref{divu-L^2-estimate} and choosing $\epsilon>0$ small enough,  we can get the estimate \eqref{u-L^4-close1}.
  To prove the second estimate \eqref{u-L^4-close2}, due to \eqref{viscosity-condition}, it holds that $\lambda+\mu>0$. Then, we can modify a term in \eqref{L^4-estimate-1}
  into
  \begin{equation*}
    8(\lambda+\mu)\int\div u|u|u\cdot\nabla|u|\leq 4(\lambda+\mu)\int|u|^2|\div u|^2+4(\lambda+\mu)\int|u|^2|\nabla|u||^2,
  \end{equation*}
which implies
  \begin{equation*}
  \frac{d}{dt}\int\rho|u|^4+\int4|u|^2(\mu|\nabla u|^2+(\mu-\lambda)|\nabla|u||^2)\leq 4\int P\div(|u|^2u).
  \end{equation*}
  Then, by employing similar arguments used in proving Lemma \ref{lem-u-L^4}, we obtain \eqref{u-L^4-close2}.

  Therefore, we have completed the proof of Lemma \ref{lem-u-L^4-close}.
\end{proof}

\begin{lem}\label{lem-theta-rho-L^2}
Under the assumptions of Theorem \ref{thm-blowup-FCNS} and \eqref{assumption1}, for any small $\epsilon>0$, it holds that
\begin{equation}\label{theta-rho-L^2-estimate}
  \begin{aligned}
  &\quad\frac{d}{dt}\int( c_v\rho|\theta-\tilde{\theta}|^2+|\rho-\tilde{\rho}|^2)+\kappa\int|\nabla\theta|^2\\
  &\leq C_{\epsilon}\left(1+\norm{\theta-\tilde{\theta}}_{L^{r,\infty}}^s\right)\left(\norm{\nabla u}_{L^2}^2+\norm{\sqrt{\rho}(\theta-\tilde{\theta})}_{L^2}^2+\norm{\rho-\tilde{\rho}}_{L^2}^2+1\right)
  +\epsilon\norm{\sqrt{\rho}\dot{u}}_{L^2}^2.
\end{aligned}
\end{equation}
\end{lem}

\begin{proof}
  Multiplying $\eqref{FCNS-eq}_3$ by $\theta-\tilde{\theta}$, and then integrating the resultant equation over $\Omega$, we have from integration by parts that
  \begin{equation}\label{theta-L^2-estimate1}
    \begin{aligned}
    &\quad\frac{1}{2}\frac{d}{dt}\int c_v\rho|\theta-\tilde{\theta}|^2+\kappa\int|\nabla\theta|^2\\
    &=-\int\rho(\theta-\tilde{\theta})^2\div u-\tilde{\theta}\int\rho(\theta-\tilde{\theta})\div u
    +2\mu\int|\mathcal{D}(u)|^2(\theta-\tilde{\theta})+\lambda\int(\div u)^2(\theta-\tilde{\theta})\\
    &\leq C\int\rho|\theta-\tilde{\theta}|^2+C\int\rho|\theta-\tilde{\theta}|^3+C\norm{\nabla u}_{L^2}^2+C\int|\theta-\tilde{\theta}||\nabla u|^2.
    \end{aligned}
  \end{equation}
  Employing similar arguments used in deriving \eqref{L^4-estimate-3-1} and \eqref{L^4-estimate-3-2}, we obtain
  \begin{itemize}
    \item If $r=\infty$, it holds that
    \begin{equation}\label{theta-L^2-estimate3-1}
      \int|\theta-\tilde{\theta}||\nabla u|^2\lesssim(1+\norm{\theta-\tilde{\theta}}_{L^{\infty}}^s)\norm{\nabla u}_{L^2}^2.
    \end{equation}
    \item If $\frac{3}{2}<r<\infty$, due to \eqref{Lorentz-Holder} and \eqref{Lorentz-interpolation}, we have
        \begin{equation}\label{theta-L^2-estimate3-2}
          \begin{aligned}
          \int|\theta-\tilde{\theta}||\nabla u|^2&\lesssim\norm{\theta-\tilde{\theta}}_{L^{r,\infty}}\norm{\nabla u}_{L^{\frac{2r}{r-1},2}}^2\\
          &\lesssim\norm{\theta-\tilde{\theta}}_{L^{r,\infty}}\norm{\nabla u}_{L^{\frac{2r_1}{r_1-1}}}\norm{\nabla u}_{L^{\frac{2r_2}{r_2-1}}}\\
          &\lesssim\norm{\theta-\tilde{\theta}}_{L^{r,\infty}}\norm{\nabla u}_{L^2}^{1-\frac{3}{2r_1}}\norm{\nabla u}_{L^6}^{\frac{3}{2r_1}}\norm{\nabla u}_{L^2}^{1-\frac{3}{2r_2}}\norm{\nabla u}_{L^6}^{\frac{3}{2r_2}}\\
          &\lesssim\norm{\theta-\tilde{\theta}}_{L^{r,\infty}}\norm{\nabla u}_{L^2}^{2-\frac{3}{r}}\norm{\nabla u}_{L^6}^{\frac{3}{r}},
          \end{aligned}
        \end{equation}
        for any $\frac{3}{2}<r_1<r<r_2<\infty$ with $\frac{2}{r}=\frac{1}{r_1}+\frac{1}{r_2}$.
        Here, we divide into two cases to deal with the term $\norm{\nabla u}_{L^6}$:
        \begin{enumerate}[(1)]
          \item For $\Omega=\mathbb{R}^3$ or bounded doman $\Omega$ with Navier-slip boundary condition \eqref{Navier-slip-condition}, by the effective viscous flux $G$, we have from \eqref{Hodge-decomposition1}, \eqref{Hodge-decomposition2}, \eqref{GW-estimate1} and \eqref{GW-estimate2} that
              \begin{equation}\label{gradient-u-L^6-1}
                \begin{aligned}
                \norm{\nabla u}_{L^6}&\lesssim\norm{G}_{L^6}+\norm{w}_{L^6}+\norm{\rho\theta-\tilde{\rho}\tilde{\theta}}_{L^6}\\
                &\lesssim\norm{\nabla G}_{L^2}+\norm{\nabla w}_{L^2}+\norm{\nabla u}_{L^2}+\norm{\rho\theta-\tilde{\rho}\tilde{\theta}}_{L^2}+\norm{\rho\theta-\tilde{\rho}\tilde{\theta}}_{L^6}\\
                &\lesssim\norm{\rho\dot{u}}_{L^2}+\norm{\nabla u}_{L^2}+\norm{\rho\theta-\tilde{\rho}\tilde{\theta}}_{L^6}\\
                &\lesssim\norm{\rho\dot{u}}_{L^2}+\norm{\nabla u}_{L^2}+\norm{\theta-\tilde{\theta}}_{L^6}+\norm{\rho-\tilde{\rho}}_{L^6}\\
                &\lesssim\norm{\sqrt{\rho}\dot{u}}_{L^2}+\norm{\nabla u}_{L^2}+\norm{\sqrt{\rho}(\theta-\tilde{\theta})}_{L^2}+\norm{\nabla\theta}_{L^2}+\norm{\rho-\tilde{\rho}}_{L^2}+1.
                \end{aligned}
              \end{equation}
              Here $\norm{\rho\theta-\tilde{\rho}\tilde{\theta}}_{L^2}$ only appears in the case of bounded $\Omega$ and thus can be absorbed by $\norm{\rho\theta-\tilde{\rho}\tilde{\theta}}_{L^6}$. In the last inequality, we have used \eqref{L2-control-lem} for the bounded domain case.
          \item If $\Omega$ is bounded with Dirichlet boundary condition \eqref{Dirichlet-condition}, by the decomposition $u=h+g$ and the estimate \eqref{h-g-estimate-Dirichlet}, we get
              \begin{equation}\label{gradient-u-L^6-2}
                \begin{aligned}
                \norm{\nabla u}_{L^6}&\leq\norm{\nabla h}_{L^6}+\norm{\nabla g}_{L^6}\lesssim\norm{\rho\theta-\tilde{\rho}\tilde{\theta}}_{L^6}+\norm{\nabla g}_{H^1}\\
                &\lesssim\norm{\theta-\tilde{\theta}}_{L^6}+\norm{\rho-\tilde{\rho}}_{L^6}+\norm{\rho\dot{u}}_{L^2}\\
                &\lesssim\norm{\sqrt{\rho}(\theta-\tilde{\theta})}_{L^2}+\norm{\nabla\theta}_{L^2}+\norm{\rho-\tilde{\rho}}_{L^2}+1+\norm{\sqrt{\rho}\dot{u}}_{L^2},
                \end{aligned}
              \end{equation}
              where we have used \eqref{L2-control-lem}.
        \end{enumerate}
        Plugging \eqref{gradient-u-L^6-1} and \eqref{gradient-u-L^6-2} into \eqref{theta-L^2-estimate3-2}, we obtain
        \begin{equation}\label{theta-L^2-estimate3-3}
          \begin{aligned}
          C\int|\theta-\tilde{\theta}||\nabla u|^2&\leq C_{\epsilon}\left(1+\norm{\theta-\tilde{\theta}}_{L^{r,\infty}}^s\right)\left(\norm{\nabla u}_{L^2}^2+\norm{\rho-\tilde{\rho}}_{L^2}^2+\norm{\sqrt{\rho}(\theta-\tilde{\theta})}_{L^2}^2+1\right)\\
          &\quad+\epsilon\left(\norm{\sqrt{\rho}\dot{u}}_{L^2}^2+\norm{\nabla\theta}_{L^2}^2\right).
          \end{aligned}
        \end{equation}
  \end{itemize}

  Similarly, for the term $\int\rho|\theta-\tilde{\theta}|^3$ we have
  \begin{itemize}
    \item If $r=\infty$,
    \begin{equation}\label{theta-L^3-estimate1}
      \int\rho|\theta-\tilde{\theta}|^3\leq(1+\norm{\theta-\tilde{\theta}}_{L^{\infty}}^s)\int\rho|\theta-\tilde{\theta}|^2.
    \end{equation}
    \item If $\frac{3}{2}<r<\infty$,
    \begin{equation}\label{theta-L^3-estimate2}
    \begin{aligned}
      \int\rho|\theta-\tilde{\theta}|^3
      &\lesssim\norm{\theta-\tilde{\theta}}_{L^{r,\infty}}\norm{\sqrt{\rho}|\theta-\tilde{\theta}|}_{L^2}^{2-\frac{3}{r}}\norm{\sqrt{\rho}|\theta-\tilde{\theta}|}_{L^6}^{\frac{3}{r}}\\
      &\lesssim\norm{\theta-\tilde{\theta}}_{L^{r,\infty}}\norm{\sqrt{\rho}|\theta-\tilde{\theta}|}_{L^2}^{2-\frac{3}{r}}(\norm{\nabla\theta}_{L^2}+\norm{\sqrt{\rho}|\theta-\tilde{\theta}|}_{L^2})^{\frac{3}{r}}\\
      &\leq C_{\epsilon}(1+\norm{\theta-\tilde{\theta}}_{L^{r,\infty}}^s)\norm{\sqrt{\rho}|\theta-\tilde{\theta}|}_{L^2}^2+\epsilon\norm{\nabla\theta}_{L^2}^2.
    \end{aligned}
    \end{equation}
  \end{itemize}
  Thus we get
  \begin{equation}\label{theta-L^3-estimate3}
    \int\rho|\theta-\tilde{\theta}|^3\leq C_{\epsilon}(1+\norm{\theta-\tilde{\theta}}_{L^{r,\infty}}^s)\norm{\sqrt{\rho}|\theta-\tilde{\theta}|}_{L^2}^2+\epsilon\norm{\nabla\theta}_{L^2}^2.
  \end{equation}

  Combining all estimates above and choosing $\epsilon>0$ sufficiently small, we conclude that
  \begin{equation}\label{theta-L^2-estimate4}
    \begin{aligned}
    &\quad\frac{d}{dt}\int c_v\rho|\theta-\tilde{\theta}|^2+\kappa\int|\nabla\theta|^2\\
    &\leq C_{\epsilon}\left(1+\norm{\theta-\tilde{\theta}}_{L^{r,\infty}}^s\right)\left(\norm{\nabla u}_{L^2}^2+\norm{\rho-\tilde{\rho}}_{L^2}^2+\norm{\sqrt{\rho}(\theta-\tilde{\theta})}_{L^2}^2+1\right)
    +\epsilon\norm{\sqrt{\rho}\dot{u}}_{L^2}^2 .
    \end{aligned}
  \end{equation}

  Next, we derive the estimate on $\norm{\rho-\tilde{\rho}}_{L^2}$. The term $\norm{\rho-\tilde{\rho}}_{L^2}$\footnote{If $\Omega$ is bounded, then this term is obviously bounded due to the mass conservation $\rho\in L^1$ and the assumption \eqref{assumption1}.} is obviously bounded if $\Omega$ is bounded. Therefore, we only need to consider the case $\Omega=\mathbb{R}^3$. To do this, we first notice that
  \begin{equation}\label{rho-L^2-eq}
    [(\rho-\tilde{\rho})^2]_t+\div[(\rho-\tilde{\rho})^2u]+(\rho-\tilde{\rho})^2\div u+2\tilde{\rho}(\rho-\tilde{\rho})\div u=0.
  \end{equation}
  Integrating the above equation over $\Omega$ and using \eqref{assumption1}, we have
  \begin{equation}\label{rho-L^2-estimate}
    \begin{aligned}
    \frac{d}{dt}\int|\rho-\tilde{\rho}|^2&\lesssim\norm{\nabla u}_{L^2}^2+\norm{\rho-\tilde{\rho}}_{L^2}^2.
    \end{aligned}
  \end{equation}
  Combining \eqref{theta-L^2-estimate4} and \eqref{rho-L^2-estimate}, we finish the proof of Lemma \ref{lem-theta-rho-L^2}.
\end{proof}

\subsection{Estimates on 1-order energy.}
In this subsection, we devote ourselves to deriving the estimates on 1-order derivative of velocity $u$.
\begin{lem}\label{lem-u-1rd-Dirichlet}
  Let $\Omega$ be bounded with Dirichlet boundary condition \eqref{Dirichlet-condition}. Under the assumptions of Theorem \ref{thm-blowup-FCNS} and \eqref{assumption1}, it holds that
  \begin{equation}\label{u-1rd-estimate-Dirichlet}
    \begin{aligned}
    &\quad\frac{d}{dt}\int[\mu(|\nabla u|^2+|\nabla h|^2)-2R(\rho\theta-\tilde{\rho}\tilde{\theta})\div u+(\mu+\lambda)(|\div u|^2+|\div h|^2)]+\int\rho|u_t|^2\\
    &\leq C\left(\norm{\theta-\tilde{\theta}}_{L^{r,\infty}}^s+1\right)\left(\norm{\sqrt{\rho}|u|^2}_{L^2}^2
    +\norm{\sqrt{\rho}(\theta-\tilde{\theta})}_{L^2}^2+\norm{\rho-\tilde{\rho}}_{L^2}^2+\norm{\nabla u}_{L^2}^2+1\right)\\
    &\quad+C\norm{\nabla\theta}_{L^2}^2+C\norm{u|\nabla u|}_{L^2}^2.
    \end{aligned}
  \end{equation}
\end{lem}

\begin{proof}
  Multiplying $\eqref{FCNS-eq}_2$ by $u_t$, and then integrating the resultant equation over $\Omega$, we have from integration by parts that
  \begin{equation}\label{u-1rd-estimate1}
    \begin{aligned}
    &\quad\frac{1}{2}\frac{d}{dt}\int[\mu|\nabla u|^2+(\mu+\lambda)|\div u|^2]+\int\rho|u_t|^2\\
    &=-\int\nabla P\cdot u_t-\int\rho u\cdot\nabla u\cdot u_t\\
    &\leq\frac{d}{dt}\int R(\rho\theta-\tilde{\rho}\tilde{\theta})\div u-\int P_t\div u+\frac{1}{4}\int\rho|u_t|^2+C\int|u|^2|\nabla u|^2.
    \end{aligned}
  \end{equation}
  By decomposing $u=h+g$, and applying \eqref{h-pressure-eq} and \eqref{g-material-eq}, we get
  \begin{equation}\label{u-1rd-estimate2}
    \begin{aligned}
    -\int P_t\div u&=-\int P_t\div h-\int P_t\div g\\
    &=\int Lh_t\cdot h-\int P_t\div g\\
    &=-\frac{1}{2}\frac{d}{dt}\int(\mu|\nabla h|^2+(\mu+\lambda)|\div h|^2)-\int P_t\div g.
    \end{aligned}
  \end{equation}
  Using the system $\eqref{FCNS-eq}$, we have
  \begin{equation}\label{P-t-eq}
  \begin{aligned}
    \frac{c_v}{R}P_t&=-\div[c_v(\rho\theta-\tilde{\rho}\tilde{\theta})u]-\tilde{\rho}\tilde{\theta}(c_v+R)\div u-R(\rho\theta-\tilde{\rho}\tilde{\theta})\div u\\
    &\qquad+2\mu|\mathcal{D}(u)|^2+\lambda|\div u|^2+\kappa\Delta\theta,
  \end{aligned}
  \end{equation}
  which implies
  \begin{equation}\label{u-1rd-estimate3}
  \begin{aligned}
    &\quad-\int\frac{c_v}{R}P_t\div g\\
    &=-\int c_v(\rho\theta-\tilde{\rho}\tilde{\theta})u\cdot\nabla\div g+\tilde{\rho}\tilde{\theta}(c_v+R)\int\div u\div g+\int R(\rho\theta-\tilde{\rho}\tilde{\theta})\div u\div g\\
    &\qquad-\int(2\mu|\mathcal{D}(u)|^2+\lambda|\div u|^2)\div g+\int\kappa\nabla\theta\cdot\nabla\div g\\
    &=-\int c_v(\rho\theta-\tilde{\rho}\tilde{\theta})u\cdot\nabla\div g-\tilde{\rho}\tilde{\theta}(c_v+R)\int u\cdot\nabla\div g\\
    &\quad+\int(\mu\Delta u+(\mu+\lambda)\nabla\div u-\nabla P)\cdot u\div g+\int2\mu\mathcal{D}(u):\nabla\div g\otimes u\\
    &\quad+\int\lambda u\cdot\nabla\div g\div u-\int R(\rho\theta-\tilde{\rho}\tilde{\theta})u\cdot\nabla\div g+\int\kappa\nabla\theta\cdot\nabla\div g\\
    &=-\int(c_v+R)(\rho\theta-\tilde{\rho}\tilde{\theta})u\cdot\nabla\div g-\tilde{\rho}\tilde{\theta}(c_v+R)\int u\cdot\nabla\div g+\int\kappa\nabla\theta\cdot\nabla\div g\\
    &\quad+\int\rho\dot{u}\cdot u\div g+\int2\mu\mathcal{D}(u):\nabla\div g\otimes u+\int\lambda u\cdot\nabla\div g\div u.
  \end{aligned}
  \end{equation}
  Therefore, by virtue of \eqref{h-g-estimate-Dirichlet} and Sobolev inequality, we conclude that
  \begin{equation}\label{u-1rd-estimate4}
  \begin{aligned}
    &\quad-\int P_t\div g\\
    &\lesssim \int|\rho\theta-\tilde{\rho}\tilde{\theta}||u||\nabla\div g|+\int|u||\nabla\div g|+\int\rho\dot{u}|u||\div g|+\int|\nabla u||u||\nabla\div g|+\int|\nabla\theta||\nabla\div g|\\
    &\leq C_{\epsilon}\left(\norm{u|\nabla u|}_{L^2}^2+\int\rho^2|\theta-\tilde{\theta}|^2|u|^2+\int|\rho-\tilde{\rho}|^2|u|^2+\norm{u}_{L^2}^2+\int\rho|u|^2|\div g|^2+\norm{\nabla\theta}_{L^2}^2\right)\\
    &\quad+\epsilon\norm{\nabla^2g}_{L^2}^2+\epsilon\norm{\sqrt{\rho}u_t}_{L^2}\\
    &\leq C\epsilon\norm{\sqrt{\rho}u_t}_{L^2}^2+C_{\epsilon}\left(\norm{u|\nabla u|}_{L^2}^2+\int\rho^2|\theta-\tilde{\theta}|^2|u|^2+\norm{\nabla u}_{L^2}^2+\int\rho|u|^2|\div h|^2+\norm{\nabla\theta}_{L^2}^2\right).
  \end{aligned}
  \end{equation}
  Next, we deal with the term $\displaystyle\int\rho|u|^2|\div h|^2$. Borrowing the ideas used in deriving \eqref{L^4-estimate-3-1} and \eqref{L^4-estimate-3-2} and the estimates in \eqref{h-g-estimate-Dirichlet} and \eqref{h-estimate-Lorentz}, we have
  \begin{itemize}
    \item if $r=\infty$, it holds that
    \begin{equation}\label{u-1rd-estimate5-1}
      \begin{aligned}
      \int\rho|u|^2|\div h|^2&\lesssim\norm{\sqrt{\rho}|u|^2}_{L^2}\norm{\nabla h}_{L^4}^2\\
      &\lesssim\norm{\sqrt{\rho}|u|^2}_{L^2}\norm{\rho\theta-\tilde{\rho}\tilde{\theta}}_{L^4}^2\\
      &\lesssim\left(\norm{\theta-\tilde{\theta}}_{L^{\infty}}+\tilde{\theta}\right)\left(\norm{\sqrt{\rho}|u|^2}_{L^2}^2
      +\norm{\sqrt{\rho}(\theta-\tilde{\theta})}_{L^2}^2+\norm{\rho-\tilde{\rho}}_{L^2}^2\right).
      \end{aligned}
    \end{equation}
    \item if $r\in(\frac{3}{2},\infty)$, by \eqref{Lorentz-Holder} and \eqref{Lorentz-interpolation}, we have
    \begin{equation}\label{u-1rd-estimate5-2}
      \begin{aligned}
      &\quad\int\rho|u|^2|\div h|^2\\
      &\lesssim\norm{\nabla h}_{L^{r,\infty}}\norm{\sqrt{\rho}|u|^2}_{L^{\frac{2r}{r-1},2}}\norm{\nabla h}_{L^{\frac{2r}{r-1},2}}\\
      &\lesssim\norm{R(\rho\theta-\tilde{\rho}\tilde{\theta})}_{L^{r,\infty}}\norm{\sqrt{\rho}|u|^2}_{L^{\frac{2r_1}{r_1-1}}}^{\frac{1}{2}}\norm{\sqrt{\rho}|u|^2}_{L^{\frac{2r_1}{r_1-1}}}^{\frac{1}{2}}\norm{\nabla h}_{L^{\frac{2r_1}{r_1-1}}}^{\frac{1}{2}}\norm{\nabla h}_{L^{\frac{2r_2}{r_2-1}}}^{\frac{1}{2}}\\
      &\lesssim\left(\norm{\rho(\theta-\tilde{\theta})}_{L^{r,\infty}}+\norm{\tilde{\theta}(\rho-\tilde{\rho})}_{L^{r,\infty}}\right)\norm{\sqrt{\rho}|u|^2}_{L^2}^{1-\frac{3}{2r}}\norm{\sqrt{\rho}|u|^2}_{L^6}^{\frac{3}{2r}}
      \norm{\nabla h}_{L^2}^{1-\frac{3}{2r}}\norm{\nabla h}_{L^6}^{\frac{3}{2r}}\\
      &\lesssim\left(\norm{\theta-\tilde{\theta}}_{L^{r,\infty}}^2+\tilde{\theta}\right)
      \left(\norm{\sqrt{\rho}|u|^2}_{L^2}^{2-\frac{3}{r}}\norm{|u|^2}_{L^6}^{\frac{3}{r}}
      +\norm{\rho\theta-\tilde{\rho}\tilde{\theta}}_{L^2}^{2-\frac{3}{r}}\norm{\rho\theta-\tilde{\rho}\tilde{\theta}}_{L^6}^{\frac{3}{r}}\right)\\
      &\leq C_{\epsilon}\left(\norm{\theta-\tilde{\theta}}_{L^{r,\infty}}^s+1\right)\left
      (\norm{\sqrt{\rho}|u|^2}_{L^2}^2+\norm{\sqrt{\rho}(\theta-\tilde{\theta})}_{L^2}^2+\norm{\rho-\tilde{\rho}}_{L^2}^2+1\right)\\
      &\qquad+\epsilon\left(\norm{\nabla\theta}_{L^2}^2+\norm{u|\nabla u|}_{L^2}^2\right),
      \end{aligned}
    \end{equation}
    where $\frac{3}{2}<r_1<r<r_2<\infty$ satisfy $\frac{2}{r}=\frac{1}{r_1}+\frac{1}{r_2}$, and we have used Sobolev inequality and \eqref{L2-control-lem} for the bounded domain $\Omega$ in the last inequality .
  \end{itemize}
  Plugging \eqref{theta-L^2-estimate3-1}, \eqref{theta-L^2-estimate3-3}, \eqref{u-1rd-estimate5-1} and \eqref{u-1rd-estimate5-2} into \eqref{u-1rd-estimate4}, we have
  \begin{equation*}
    \begin{aligned}
    -\int P_t\div g&\leq C_{\epsilon}\left(\norm{\theta-\tilde{\theta}}_{L^{r,\infty}}^s+1\right)\left
    (\norm{\sqrt{\rho}|u|^2}_{L^2}^2+\norm{\sqrt{\rho}(\theta-\tilde{\theta})}_{L^2}^2+\norm{\rho-\tilde{\rho}}_{L^2}^2+\norm{\nabla u}_{L^2}^2+\norm{\nabla\theta}_{L^2}^2+1\right)\\
    &\quad+C\epsilon\left(\norm{\sqrt{\rho}u_t}_{L^2}^2+\norm{u|\nabla u|}_{L^2}^2\right),
    \end{aligned}
  \end{equation*}
  which together with \eqref{u-1rd-estimate1} implies \eqref{u-1rd-estimate-Dirichlet} immediately.

 Therefore, we have completed the proof of Lemma \ref{lem-u-1rd-Dirichlet}.
\end{proof}
\begin{cor}\label{cor-u-1rd-Dirichlet}
Let $\Omega$ be bounded with Dirichlet boundary condition \eqref{Dirichlet-condition}. Under the assumptions of Theorem \ref{thm-blowup-FCNS} and \eqref{assumption1}, we have
\begin{equation}\label{cor-u-1rd-estimate-Dirichlet}
  \sup_{t\in[0,T]}\int(\rho|u|^4+\rho|\theta-\tilde{\theta}|^2+|\rho-\tilde{\rho}|^2+|\nabla u|^2)+\int_0^T\int(\rho|u_t|^2+|u|^2|\nabla u|^2+|\nabla\theta|^2)\leq C,
\end{equation}
for any $T\in(0,T^*)$.
\end{cor}

\begin{proof}
  Since the proof is similar to that of \cite{Feireisl2024}, we only sketch it here.
To begin with, multiplying \eqref{theta-rho-L^2-estimate} by a large number $M>0$ and adding it to \eqref{u-1rd-estimate-Dirichlet}, and taking $\epsilon>0$ suitably small, we obtain
  \begin{equation}\label{u-1rd-Dirichlet-transition1}
  \begin{aligned}
    &\quad\frac{d}{dt}\int F_1+\int\rho|u_t|^2+M\int|\nabla\theta|^2\\
    &\leq C\left(\norm{\theta-\tilde{\theta}}_{L^{r,\infty}}^s+1\right)\left
    (\norm{\sqrt{\rho}|u|^2}_{L^2}^2+\norm{\sqrt{\rho}(\theta-\tilde{\theta})}_{L^2}^2+\norm{\rho-\tilde{\rho}}_{L^2}^2+\norm{\nabla u}_{L^2}^2+1\right)\\
    &\qquad+C\norm{u|\nabla u|}_{L^2}^2,
  \end{aligned}
  \end{equation}
  where
  \begin{equation*}
    F_1\sim |\nabla u|^2+|\nabla h|^2+M\left(\rho|\theta-\tilde{\theta}|^2+|\rho-\tilde{\rho}|^2\right).
  \end{equation*}

  For the case $\lambda>0$,  multiplying \eqref{u-L^4-close1} by a large number $M_1>0$ and adding it into \eqref{u-1rd-Dirichlet-transition1}, and taking $\epsilon>0$ suitably small, we get
  \begin{equation}\label{u-1rd-Dirichlet-transition2}
  \begin{aligned}
    &\quad\frac{d}{dt}\int F_2+\int\rho|u_t|^2+\int|\nabla\theta|^2+\int|u|^2|\nabla u|^2\\
    &\leq C\left(\norm{\theta-\tilde{\theta}}_{L^{r,\infty}}^s+1\right)\left
    (\norm{\sqrt{\rho}|u|^2}_{L^2}^2+\norm{\sqrt{\rho}(\theta-\tilde{\theta})}_{L^2}^2+\norm{\rho-\tilde{\rho}}_{L^2}^2+\norm{\nabla u}_{L^2}^2+1\right),
  \end{aligned}
  \end{equation}
  where
  \begin{equation*}
    F_2\sim |\nabla u|^2+|\nabla h|^2+\rho|\theta-\tilde{\theta}|^2+|\rho-\tilde{\rho}|^2+\rho|u|^4,
  \end{equation*}
  and $M_1$ satisfies that
  \begin{equation*}
    \sqrt{M_1M}\gtrsim\frac{4Cc_v}{\lambda}M_1, \quad M_1\mu>2C,
  \end{equation*}
  where $C$ is the constant of $C\norm{u|\nabla u|}_{L^2}^2$ in \eqref{u-1rd-Dirichlet-transition1}. This means that we should choose $M_1$ firstly and then determine the size of $M$. Then, applying Gronwall's inequality to \eqref{u-1rd-Dirichlet-transition2} yields the desired estimate \eqref{cor-u-1rd-estimate-Dirichlet} immediately.

  For the case $\lambda\leq 0$, multiplying \eqref{u-L^4-close2} by a large number $M_2>0$, adding it into \eqref{u-1rd-Dirichlet-transition1}, and taking $\epsilon>0$ suitably small, we obtain \eqref{u-1rd-Dirichlet-transition2}. Thus, we have completed the proof of Lemma \ref{cor-u-1rd-Dirichlet}.
\end{proof}
Next, we will give the estimate on the 1-order derivative of velocity $u$ for the Cauchy problem case.
\begin{lem}\label{lem-u-1rd-Cauchy}
  Let $\Omega=\mathbb{R}^3$. Under the assumptions of Theorem \ref{thm-blowup-FCNS} and \eqref{assumption1}, we have
  \begin{equation}\label{u-1rd-estimate-Cauchy}
    \begin{aligned}
    &\quad\frac{d}{dt}\int[\mu|\nabla u|^2-2R(\rho\theta-\tilde{\rho}\tilde{\theta})\div u+(\mu+\lambda)|\div u|^2+\frac{1}{2\mu+\lambda}|R(\rho\theta-\tilde{\rho}\tilde{\theta})|^2]+\int\rho|u_t|^2\\
    &\leq C\left(1+\norm{\theta-\tilde{\theta}}_{L^{r,\infty}}^s\right)\int\left(\rho|u|^4+\rho|\theta-\tilde{\theta}|^2\right)+C\int\rho|u|^2
    +C\int|\nabla\theta|^2+C\int|u|^2|\nabla u|^2.
    \end{aligned}
  \end{equation}
  In particular, if $\tilde{\theta}=0$, the term $\displaystyle\int\rho|u|^2$ in the right-hand side will vanish.
\end{lem}

\begin{proof}
  Similar to \eqref{u-1rd-estimate1}, we have
  \begin{equation}\label{u-1rd-estimate-Cauchy1}
    \begin{aligned}
    &\quad\frac{1}{2}\frac{d}{dt}\int(\mu|\nabla u|^2+(\mu+\lambda)|\div u|^2)+\int\rho|u_t|^2\\
    &=-\int\nabla P\cdot u_t-\int\rho u\cdot\nabla u\cdot u_t\\
    &\leq\frac{d}{dt}\int R(\rho\theta-\tilde{\rho}\tilde{\theta})\div u-\int P_t\div u+\frac{1}{4}\int\rho|u_t|^2+C\int|u|^2|\nabla u|^2\\
    &\leq\frac{d}{dt}\int R(\rho\theta-\tilde{\rho}\tilde{\theta})\div u-\frac{1}{2(2\mu+\lambda)}\frac{d}{dt}\int|R(\rho\theta-\tilde{\rho}\tilde{\theta})|^2-\frac{1}{2\mu+\lambda}\int P_tG\\
    &\qquad+\frac{1}{4}\int\rho|u_t|^2+C\int|u|^2|\nabla u|^2,
    \end{aligned}
  \end{equation}
  where $G=(2\mu+\lambda)\div u-P(\rho,\theta)+P(\tilde{\rho},\tilde{\theta})$.

  Analogous to \eqref{u-1rd-estimate3}, by noticing that
  \begin{equation*}
  \begin{aligned}
    \frac{c_v}{R}P_t&=-\div[c_v(\rho\theta-\tilde{\rho}\tilde{\theta})u]-\tilde{\rho}\tilde{\theta}(c_v+R)\div u-R(\rho\theta-\tilde{\rho}\tilde{\theta})\div u\\
    &\qquad+2\mu|\mathcal{D}(u)|^2+\lambda|\div u|^2+\kappa\Delta\theta,
  \end{aligned}
  \end{equation*}
  we obtain
  \begin{equation}\label{u-1rd-estimate-Cauchy2}
  \begin{aligned}
    -\int\frac{c_v}{R}P_tG
    &=-\int c_v(\rho\theta-\tilde{\rho}\tilde{\theta})u\cdot\nabla G+\tilde{\rho}\tilde{\theta}(c_v+R)\int\div uG+\int R(\rho\theta-\tilde{\rho}\tilde{\theta})\div uG\\
    &\qquad-\int(2\mu|\mathcal{D}(u)|^2+\lambda|\div u|^2)G+\int\kappa\nabla\theta\cdot\nabla G\\
    &=-\int c_v(\rho\theta-\tilde{\rho}\tilde{\theta})u\cdot\nabla G+\tilde{\rho}\tilde{\theta}(c_v+R)\int\div uG\\
    &\quad+\int\left(\mu\Delta u+(\mu+\lambda)\nabla\div u-\nabla P\right)\cdot uG+\int2\mu\mathcal{D}(u):\nabla G\otimes u\\
    &\quad+\int\lambda u\cdot\nabla G\div u-\int R(\rho\theta-\tilde{\rho}\tilde{\theta})u\cdot\nabla G+\int\kappa\nabla\theta\cdot\nabla G\\
    &=-(c_v+R)\int\left(\rho(\theta-\tilde{\theta})+\tilde{\theta}\rho\right)u\cdot\nabla G+\int\kappa\nabla\theta\cdot\nabla G\\
    &\quad+\int\rho\dot{u}\cdot uG+\int2\mu\mathcal{D}(u):\nabla G\otimes u+\int\lambda u\cdot\nabla G\div u.
  \end{aligned}
  \end{equation}
  Then, we deduce from \eqref{GW-estimate1} and Sobolev inequality that
  \begin{equation}\label{u-1rd-estimate-Cauchy3}
    \begin{aligned}
    &\quad-\frac{1}{2\mu+\lambda}\int P_tG\\
    &\leq\epsilon\norm{\sqrt{\rho}u_t}_{L^2}^2+\epsilon\norm{\nabla G}_{L^2}^2+C_{\epsilon}\left(\int\rho^2|\theta-\tilde{\theta}|^2|u|^2+\int\rho^2|u|^2+\norm{\nabla\theta}_{L^2}^2+\int\rho|u|^2|G|^2\right)\\
    &\qquad+C\int\rho|u|^2|\nabla u|^2\\
    &\leq C\epsilon\int\rho|u_t|^2+C_{\epsilon}\int\left(\rho^2|\theta-\tilde{\theta}|^2|u|^2+\rho|u|^2+|\nabla \theta|^2+|u|^2|\nabla u|^2\right).
    \end{aligned}
  \end{equation}
  Combining \eqref{L^4-estimate-3-1}, \eqref{L^4-estimate-3-2}, \eqref{u-1rd-estimate-Cauchy3} with \eqref{u-1rd-estimate-Cauchy1}, we obtain
  \begin{equation*}
    \begin{aligned}
    &\quad\frac{d}{dt}\int\left(\mu|\nabla u|^2-2R(\rho\theta-\tilde{\rho}\tilde{\theta})\div u+(\mu+\lambda)|\div u|^2+\frac{1}{2\mu+\lambda}|R(\rho\theta-\tilde{\rho}\tilde{\theta})|^2\right)+\int\rho|u_t|^2\\
    &\leq C\left(1+\norm{\theta-\tilde{\theta}}_{L^{r,\infty}}^s\right)\int\left(\rho|u|^4+\rho|\theta-\tilde{\theta}|^2\right)+C\int\rho|u|^2
    +C\int|\nabla\theta|^2+C\int|u|^2|\nabla u|^2.
    \end{aligned}
  \end{equation*}
  Thus, we have completed the proof of Lemma \ref{lem-u-1rd-Cauchy}.
\end{proof}
Except for the above estimate \eqref{u-1rd-estimate-Cauchy}, we still need a control on $\displaystyle\int\rho|u|^2$ to close the 1-order energy.
\begin{lem}\label{lem-u-L^2-estimate}
  Let $\Omega=\mathbb{R}^3$. Under the assumptions of Theorem \ref{thm-blowup-FCNS} and \eqref{assumption1}, we have
  \begin{equation}\label{u-L^2-estimate}
    \frac{d}{dt}\int\rho|u|^2+\mu\int|\nabla u|^2\leq C\int(\rho|\theta-\tilde{\theta}|^2+|\rho-\tilde{\rho}|^2).
  \end{equation}
\end{lem}

\begin{proof}
  Multiplying $\eqref{FCNS-eq}_2$ by $u$ and integrating the resultant equation over $\Omega$, we obtain
  \begin{equation*}
    \begin{aligned}
    &\quad\frac{d}{dt}\int\rho|u|^2+2\int[\mu|\nabla u|^2+(\mu+\lambda)|\div u|^2]\\
    &=2\int R(\rho\theta-\tilde{\rho}\tilde{\theta})\div u\\
    &\leq\mu\int|\nabla u|^2+C\int\rho|\theta-\tilde{\theta}|^2+C\int|\rho-\tilde{\rho}|^2,
    \end{aligned}
  \end{equation*}
  which implies \eqref{u-L^2-estimate}. Thus, the proof of Lemma \ref{lem-u-L^2-estimate} is completed.
\end{proof}

Similar to Corollary \ref{cor-u-1rd-Dirichlet}, by combining Lemma \ref{lem-u-L^4-close}, \ref{lem-theta-rho-L^2}, \ref{lem-u-1rd-Cauchy} and \ref{lem-u-L^2-estimate}, we have the following corollary.

\begin{cor}\label{cor-u-1rd-Cauchy}
  Let $\Omega=\mathbb{R}^3$. Under the assumptions of Theorem \ref{thm-blowup-FCNS} and \eqref{assumption1}, it holds that
  \begin{equation}\label{cor-u-1rd-estimate-Cauchy}
  \sup_{t\in[0,T]}\int(\rho|u|^2+\rho|u|^4+\rho|\theta-\tilde{\theta}|^2+|\rho-\tilde{\rho}|^2+|\nabla u|^2)+\int_0^T\int(\rho|u_t|^2+|u|^2|\nabla u|^2+|\nabla\theta|^2)\leq C,
\end{equation}
for any $T\in(0,T^*)$.
\end{cor}

Next, we  modify the proof of Lemma \ref{lem-u-1rd-Cauchy} for the case $\Omega=\mathbb{R}^3$ to get the following lemma for bounded domain $\Omega$ with Navier-slip boundary condition \eqref{Navier-slip-condition}.

\begin{lem}\label{lem-u-1rd-Navier}
  Let $\Omega$ be bounded with Navier-slip boundary condition \eqref{Navier-slip-condition}. Under the assumptions of Theorem \ref{thm-blowup-FCNS} and \eqref{assumption1}, we have
  \begin{equation}\label{u-1rd-estimate-Navier}
    \begin{aligned}
    &\quad\frac{d}{dt}\int[\mu|\curl u|^2-2R(\rho\theta-\tilde{\rho}\tilde{\theta})\div u+(2\mu+\lambda)|\div u|^2+\frac{1}{2\mu+\lambda}|R(\rho\theta-\tilde{\rho}\tilde{\theta})|^2]+\int\rho|u_t|^2\\
    &\leq C\left(1+\norm{\theta-\tilde{\theta}}_{L^{r,\infty}}^s\right)\int\left(\rho|u|^4+\rho|\theta-\tilde{\theta}|^2\right)+C\int\left(|\rho-\tilde{\rho}|^2+|\nabla u|^2\right)\\
    &\qquad+C\int|\nabla\theta|^2+C\int|u|^2|\nabla u|^2+C.
    \end{aligned}
  \end{equation}
\end{lem}

\begin{proof}
  If $\Omega$ is bounded with Navier-slip boundary condition, the equality \eqref{u-1rd-estimate-Cauchy2} will involve a boundary term as
  \begin{equation}\label{u-1rd-boundary}
    -\int_{\pa\Omega}2\mu\mathcal{D}(u)n\cdot uG.
  \end{equation}
  Recalling that
  \begin{equation*}
    u\cdot n=0,\quad \curl u\times n=0 \textrm{ on }\pa\Omega,
  \end{equation*}
  which is equivalent to (see \cite{Cai-Li2023})
  \begin{equation*}
    u\cdot n=0,\quad (\mathcal{D}(u)n)_{\tau}=-\kappa_{\tau}u_{\tau} \textrm{ on }\pa\Omega,
  \end{equation*}
  where $\kappa_{\tau}$ is the corresponding principal curvature of $\pa\Omega$ and $u_{\tau}$ represents the projection of tangent plane of the vector $u$ on $\pa\Omega$. Then, we have from \eqref{trace-inequality}, \eqref{GW-estimate2} and \eqref{L2-control-lem} that
  \begin{equation}\label{u-1rd-estimate-boundary}
    \begin{aligned}
    -\int_{\pa\Omega}2\mu\mathcal{D}(u)n\cdot uG&=-\int_{\pa\Omega}2\mu(\mathcal{D}(u)n)_{\tau}\cdot u_{\tau}G=\int_{\pa\Omega}2\mu\kappa_{\tau}|u_{\tau}|^2G\\
    &\leq C\norm{G}_{L^2(\pa\Omega)}\norm{|u|^2}_{L^2(\pa\Omega)}
    \lesssim\norm{G}_{L^2}^{\frac{1}{2}}\norm{G}_{H^1}^{\frac{1}{2}}\norm{|u|^2}_{L^2}^{\frac{1}{2}}\norm{|u|^2}_{H^1}^{\frac{1}{2}}\\
    &\leq\epsilon\norm{\nabla G}_{L^2}^2+\epsilon\norm{u|\nabla u|}_{L^2}^2+C_{\epsilon}(\norm{G}_{L^2}^2+\norm{|u|^2}_{L^2}^2)\\
    &\leq C\epsilon\int\rho|u_t|^2+C_{\epsilon}\int\left(|\nabla u|^2+\rho|\theta-\tilde{\theta}|^2+|\rho-\tilde{\rho}|^2+\rho|u|^4+|u|^2|\nabla u|^2\right).
    \end{aligned}
  \end{equation}
 Next, we deal with $\displaystyle\int\rho|u|^2$ in the inequality \eqref{u-1rd-estimate-Cauchy3} as follows:
  \begin{equation}\label{u-L^2-Navier-control}
    \int\rho|u|^2\leq\norm{\sqrt{\rho}|u|^2}_{L^2}\norm{\sqrt{\rho}}_{L^2}\leq C\left(1+\int\rho|u|^4\right).
  \end{equation}
  Thus, following the procedure in Lemma \ref{lem-u-1rd-Cauchy}, we deduce \eqref{u-1rd-estimate-Navier}, and thus complete the
  proof of Lemma \ref{lem-u-1rd-Navier}.
\end{proof}

The following corollary can be deduced from Lemma \ref{lem-u-L^4-close}, \ref{lem-theta-rho-L^2} and \ref{lem-u-1rd-Navier} immediately just as the same in Corollary \ref{cor-u-1rd-Dirichlet}.
\begin{cor}\label{cor-u-1rd-Navier}
  Let $\Omega$ be bounded with Navier-slip boundary condition \eqref{Navier-slip-condition}. Under the assumptions of Theorem \ref{thm-blowup-FCNS} and \eqref{assumption1}, it holds that
  \begin{equation}\label{cor-u-1rd-estimate-Navier}
  \sup_{t\in[0,T]}\int(\rho|u|^4+\rho|\theta-\tilde{\theta}|^2+|\rho-\tilde{\rho}|^2+|\nabla u|^2)+\int_0^T\int(\rho|u_t|^2+|u|^2|\nabla u|^2+|\nabla\theta|^2)\leq C,
\end{equation}
for any $T\in(0,T^*)$.
\end{cor}

\subsection{Proof of Theorem \ref{thm-blowup-FCNS}.} With the key 0-order and 1-order energy estimates in hand, we can modify the methods of
\cite{Feireisl2024} to get the crucial uniform energy estimate \eqref{goal-inequality}
for three different cases. It should be mentioned that for the bounded domain $\Omega$ with Navier-slip boundary condition \eqref{Navier-slip-condition}, we need to deal with some trouble boundary terms similar as \eqref{u-1rd-boundary} carefully. Since this process is standard and can be found in \cite{Cai-Li2023} for some new ingredients, we omit the details here.
Considering the available local existence results stated in Remark \ref{rem-local-existence}, it is easy to verify that \eqref{goal-inequality} indicates the strong solution can be extended beyond $T^*$ (see \cite{Huang-Li-Wang2013} for instance), which means $T^*$ is not the maximal existence time, and thus leads to a contradiction.
Therefore, we complete the proof of Theorem \ref{thm-blowup-FCNS}.

\section{Proof of Theorem \ref{thm-blowup-FCNS-velocity}}\label{sect-proof-theorem-FCNS-u}
Assume that $T^*<\infty$ and there exist constants $r\in(3,\infty]$, $s\in[2,\infty]$ satisfying \eqref{index-blowup2} such that
\begin{equation}\label{assumption2}
  \norm{\rho}_{L^{\infty}(0,T;L^{\infty})}+\norm{u}_{L^s(0,T;L^r_w)}\leq M^*<\infty,
\end{equation}
for any $T\in(0,T^*)$.

Then, we derive the following key estimate on $\displaystyle\int\rho|\theta-\tilde{\theta}|^2$.

\begin{lem}\label{lem-theta-L^2-blowup-u}
  Under the assumptions of Theorem \ref{thm-blowup-FCNS-velocity} and \eqref{assumption2}, it holds that for any small $\epsilon_1>0$,
  \begin{equation}\label{theta-L^2-estimate-blowup-u}
  \begin{aligned}
  &\quad\frac{d}{dt}\int c_v\rho|\theta-\tilde{\theta}|^2+\kappa\int|\nabla\theta|^2\\
  &\leq C_{\epsilon_1}\left(1+\norm{u}_{L^{r,\infty}}^s\right)\left(\norm{\nabla u}_{L^2}^2+\norm{\sqrt{\rho}u}_{L^2}^2+\norm{\sqrt{\rho}(\theta-\tilde{\theta})}_{L^2}^2+\norm{\rho-\tilde{\rho}}_{L^2}^2+1\right)
  +\epsilon_1\norm{\sqrt{\rho}\dot{u}}_{L^2}^2.
  \end{aligned}
  \end{equation}

\end{lem}

\begin{proof}
  Multiplying $\eqref{FCNS-eq}_3$ by $\theta-\tilde{\theta}$ and integrating the resultant equation over $\Omega$, we have
  \begin{equation}\label{theta-L^2-estimate-blowup-u1}
    \begin{aligned}
    &\quad\frac{1}{2}\frac{d}{dt}\int c_v\rho|\theta-\tilde{\theta}|^2+\kappa\int|\nabla\theta|^2\\
    &=\int(2\mu|\mathcal{D}(u)|^2+\lambda|\div u|^2-P\div u)(\theta-\tilde{\theta})\\
    &=-\int(\mu\Delta u+(\mu+\lambda)\nabla\div u-\nabla P)\cdot u(\theta-\tilde{\theta})-\int2\mu\mathcal{D}(u):\nabla\theta\otimes u-\int(\lambda\div u-P)u\cdot\nabla\theta\\
    &=-\int\rho\dot{u}\cdot u(\theta-\tilde{\theta})-\int2\mu\mathcal{D}(u):\nabla\theta\otimes u-\int\lambda\div uu\cdot\nabla\theta+\int [R\rho(\theta-\tilde{\theta})+R\tilde{\theta}\rho]u\cdot\nabla\theta\\
    &\leq\epsilon\int\rho|\dot{u}|^2+\epsilon\int|\nabla\theta|^2
    +C_{\epsilon}\int\left((\rho+\rho^2)|u|^2|\theta-\tilde{\theta}|^2+|u|^2|\nabla u|^2\right)+C_{\epsilon}\tilde{\theta}\int\rho^2|u|^2\\
    &\leq\epsilon\int\rho|\dot{u}|^2+\epsilon\int|\nabla\theta|^2
    +C_{\epsilon}\int\left(\rho|u|^2|\theta-\tilde{\theta}|^2+|u|^2|\nabla u|^2\right)+C_{\epsilon}\tilde{\theta}\int\rho|u|^2.
    \end{aligned}
  \end{equation}
  For the Navier-slip boundary condition \eqref{Navier-slip-condition}, there exists a boundary term in \eqref{theta-L^2-estimate-blowup-u1} as
  \begin{equation}\label{boundary-term1}
    \begin{aligned}
    \int_{\pa\Omega}2\mu\mathcal{D}(u)n\cdot u(\theta-\tilde{\theta})&=-\int_{\pa\Omega}2\mu\kappa_{\tau}|u_{\tau}|^2(\theta-\tilde{\theta})
    \lesssim\norm{\theta-\tilde{\theta}}_{L^2(\pa\Omega)}\norm{|u|^2}_{L^2(\pa\Omega)}\\
    &\lesssim\norm{\theta-\tilde{\theta}}_{L^2}^{\frac{1}{2}}\norm{\theta-\tilde{\theta}}_{H^1}^{\frac{1}{2}}
    \norm{|u|^2}_{L^2}^{\frac{1}{2}}\norm{|u|^2}_{H^1}^{\frac{1}{2}}\\
    &\leq\epsilon\left(\norm{\nabla\theta}_{L^2}^2+\norm{u|\nabla u|}_{L^2}^2\right)+C_{\epsilon}\left
    (\norm{\sqrt{\rho}(\theta-\tilde{\theta})}_{L^2}^2+\norm{|u|^2}_{L^2}^2\right),
    \end{aligned}
  \end{equation}
  where we have used the similar argument as in \eqref{u-1rd-estimate-boundary}. Here, we need to bound $\displaystyle\int|u|^4$ on bounded domain $\Omega$ with Navier-slip boundary condition \eqref{Navier-slip-condition}.\\
  Noticing that if $\tilde{\theta}=0$, there is no need to deal with $\displaystyle\int\rho|u|^2$ just as in \cite{Huang-Li-Wang2013}.
We first deal with $I_1:=\displaystyle\int\rho|u|^2|\theta-\tilde{\theta}|^2$. There exist two cases for the index of blowup criterion \eqref{index-blowup2} as follows:
  \begin{itemize}
    \item If $r=\infty$, then $L^{\infty}_w=L^{\infty}$. By virtue of \eqref{assumption2}, we have
    \begin{equation}\label{theta-L^2-estimate-blowup-u2-1}
      I_1\lesssim\norm{u}_{L^{\infty}}^2\int\rho|\theta-\tilde{\theta}|^2\lesssim(\norm{u}_{L^{\infty}}^s+1)\int\rho|\theta-\tilde{\theta}|^2.
    \end{equation}
    \item If $\frac{2}{s}+\frac{3}{r}\leq1$ with $r\in(3,\infty)$, by \eqref{assumption2}, \eqref{Lorentz-Holder} and \eqref{Lorentz-interpolation}, we get
    \begin{equation}\label{theta-L^2-estimate-blowup-u2-2}
    \begin{aligned}
      I_1&\lesssim\norm{u}_{L^{r,\infty}}^2\norm{\sqrt{\rho}(\theta-\tilde{\theta})}_{L^{\frac{2r}{r-2},2}}^2\\
      &\lesssim\norm{u}_{L^{r,\infty}}^2\norm{\sqrt{\rho}(\theta-\tilde{\theta})}_{L^{\frac{2r_1}{r_1-2}}}\norm{\sqrt{\rho}(\theta-\tilde{\theta})}_{L^{\frac{2r_2}{r_2-2}}}\\
      &\lesssim\norm{u}_{L^{r,\infty}}^2\norm{\sqrt{\rho}(\theta-\tilde{\theta})}_{L^2}^{1-\frac{3}{r_1}}\norm{\sqrt{\rho}(\theta-\tilde{\theta})}_{L^6}^{\frac{3}{r_1}}
      \norm{\sqrt{\rho}(\theta-\tilde{\theta})}_{L^2}^{1-\frac{3}{r_2}}\norm{\sqrt{\rho}(\theta-\tilde{\theta})}_{L^6}^{\frac{3}{r_2}}\\
      &\lesssim\norm{u}_{L^{r,\infty}}^2\norm{\sqrt{\rho}(\theta-\tilde{\theta})}_{L^2}^{2-\frac{6}{r}}\norm{\sqrt{\rho}(\theta-\tilde{\theta})}_{L^6}^{\frac{6}{r}}\\
      &\leq C_{\epsilon}\norm{u}_{L^{r,\infty}}^{\frac{2}{1-\frac{3}{r}}}\norm{\sqrt{\rho}(\theta-\tilde{\theta})}_{L^2}^2+\epsilon\norm{\theta-\tilde{\theta}}_{L^6}^2\\
      &\leq C_{\epsilon}\left(1+\norm{u}_{L^{r,\infty}}^s\right)\norm{\sqrt{\rho}(\theta-\tilde{\theta})}_{L^2}^2+C\epsilon\norm{\nabla\theta}_{L^2}^2,
    \end{aligned}
    \end{equation}
    where $3<r_1<r<r_2<\infty$ satisfy $\frac{2}{r}=\frac{1}{r_1}+\frac{1}{r_2}$, and we have used Sobolev inequality and \eqref{L2-control-lem} for the bounded domain $\Omega$.
  \end{itemize}
   To bound $\displaystyle\int|u|^4$ in \eqref{boundary-term1}, we follow the steps in estimating $I_1$ in \eqref{theta-L^2-estimate-blowup-u2-1} and \eqref{theta-L^2-estimate-blowup-u2-2} to get
   \begin{itemize}
     \item If $r=\infty$, then it holds that
     \begin{equation}\label{theta-u-L^4-estimate1}
       \begin{aligned}
       \int|u|^4\lesssim\norm{u}_{L^{\infty}}^2\int|u|^2\lesssim\left(1+\norm{u}_{L^{\infty}}^s\right)\left(\int\rho|u|^2+\int|\nabla u|^2\right).
       \end{aligned}
     \end{equation}
     \item If $r<\infty$, similar to \eqref{theta-L^2-estimate-blowup-u2-2}, we have
     \begin{equation}\label{theta-u-L^4-estimate2}
       \begin{aligned}
       \int|u|^4&\lesssim\norm{u}_{L^{r,\infty}}^2\norm{u}_{L^{\frac{2r}{r-2},2}}\lesssim\norm{u}_{L^{r,\infty}}^2\norm{u}_{L^{\frac{2r_1}{r_1-2}}}\norm{u}_{L^{\frac{2r_2}{r_2-2}}}\\
       &\lesssim\norm{u}_{L^{r,\infty}}^2\norm{u}_{L^2}^{2(1-\frac{3}{r})}\norm{u}_{L^6}^{\frac{6}{r}}\\
       &\leq C\left(1+\norm{u}_{L^{r,\infty}}^s\right)\left(\int\rho|u|^2+\norm{\nabla u}_{L^2}^2\right),
       \end{aligned}
     \end{equation}
   \end{itemize}
  where we have used Sobolev inequality and \eqref{L2-control-lem} for the bounded domain $\Omega$.

  The estimate on $I_2:=\displaystyle\int|u|^2|\nabla u|^2$ is much more complicated than that of $\displaystyle\int\rho|u|^2|\theta-\tilde{\theta}|^2$. Here, we will make more delicate estimates by decomposition of velocity $u=h+g$ in $\Omega=\mathbb{R}^3$ or bounded domain $\Omega$ with Dirichlet boundary condition \eqref{Dirichlet-condition}.
Employing similar arguments used in deriving \eqref{theta-L^2-estimate-blowup-u2-1} and \eqref{theta-L^2-estimate-blowup-u2-2}, we obtain
  \begin{itemize}
    \item If $r=\infty$, then
    \begin{equation}\label{theta-L^2-estimate-blowup-u3-1}
      I_2=\int|u|^2|\nabla u|^2\lesssim\left(1+\norm{u}_{L^{\infty}}^s\right)\norm{\nabla u}_{L^2}^2.
    \end{equation}
    \item If $\frac{2}{s}+\frac{3}{r}\leq1$ with $r\in(3,\infty)$, we have
        \begin{equation}\label{theta-L^2-estimate-blowup-u3-2}
          \begin{aligned}
          I_2=\int|u|^2|\nabla u|^2&\lesssim\norm{u}_{L^{r,\infty}}^2\norm{\nabla u}_{L^2}^{2-\frac{6}{r}}\norm{\nabla u}_{L^6}^{\frac{6}{r}}.
          \end{aligned}
        \end{equation}
        Here, we consider three different boundary conditions.
        \begin{enumerate}[(1)]
          \item If $\Omega=\mathbb{R}^3$, then by decomposition $u=g+h$, the elliptic estimates \eqref{h-g-estimate-Whole} and \eqref{assumption2}, we get
          \begin{equation}\label{theta-L^2-estimate-blowup-u3-2-1}
          \begin{aligned}
            \norm{\nabla u}_{L^6}&\lesssim\norm{\nabla h}_{L^6}+\norm{\nabla g}_{L^6}\\
            &\lesssim\norm{\rho\theta-\tilde{\rho}\tilde{\theta}}_{L^6}+\norm{\nabla^2g}_{L^2}\\
            &\lesssim\norm{\rho(\theta-\tilde{\theta})}_{L^6}+\tilde{\theta
            }\norm{\rho-\tilde{\rho}}_{L^6}+\norm{\rho\dot{u}}_{L^2}\\
            &\lesssim\norm{\nabla\theta}_{L^2}+\norm{\rho-\tilde{\rho}}_{L^2}+\norm{\sqrt{\rho}\dot{u}}_{L^2}+1.
          \end{aligned}
          \end{equation}
          Although $\tilde{\rho}=0$ holds in this section, we still can keep $\tilde{\rho}\geq 0$ here and in similar situations below.
          \item If $\Omega$ is bounded with Dirichlet boundary condition \eqref{Dirichlet-condition}, also by decomposition $u=g+h$, the elliptic estimates \eqref{h-g-estimate-Dirichlet}, \eqref{assumption2} and \eqref{L2-control-lem}, we have
          \begin{equation}\label{theta-L^2-estimate-blowup-u3-2-2}
          \begin{aligned}
            \norm{\nabla u}_{L^6}&\lesssim\norm{\nabla h}_{L^6}+\norm{\nabla g}_{L^6}\\
            &\lesssim\norm{\rho\theta-\tilde{\rho}\tilde{\theta}}_{L^6}+\norm{\nabla g}_{H^1}\\
            &\lesssim\norm{\rho(\theta-\tilde{\theta})}_{L^6}+\tilde{\theta
            }\norm{\rho-\tilde{\rho}}_{L^6}+\norm{\rho\dot{u}}_{L^2}\\
            &\lesssim\norm{\sqrt{\rho}(\theta-\tilde{\theta})}_{L^2}+\norm{\nabla\theta}_{L^2}+\norm{\rho-\tilde{\rho}}_{L^2}+\norm{\sqrt{\rho}\dot{u}}_{L^2}+1.
          \end{aligned}
          \end{equation}
          \item If $\Omega$ is bounded with Navier-slip boundary condition \eqref{Navier-slip-condition}, by \eqref{Hodge-decomposition2}, \eqref{GW-estimate2}, \eqref{assumption2} and \eqref{L2-control-lem}, we obtain
          \begin{equation}\label{theta-L^2-estimate-blowup-u3-2-3}
          \begin{aligned}
            \norm{\nabla u}_{L^6}&\lesssim\norm{ G}_{L^6}+\norm{w}_{L^6}+\norm{P-P(\tilde{\rho},\tilde{\theta})}_{L^6}\\
            &\lesssim\norm{\nabla u}_{L^2}+\norm{\nabla G}_{L^2}+\norm{\nabla w}_{L^2}+\norm{P-P(\tilde{\rho},\tilde{\theta})}_{L^2}+\norm{P-P(\tilde{\rho},\tilde{\theta})}_{L^6}\\
            &\lesssim\norm{\nabla u}_{L^2}+\norm{\rho\dot{u}}_{L^2}+\norm{\sqrt{\rho}(\theta-\tilde{\theta})}_{L^2}+\norm{\nabla\theta
            }_{L^2}+\norm{\rho-\tilde{\rho}}_{L^2}+1\\
            &\lesssim\norm{\nabla u}_{L^2}+\norm{\sqrt{\rho}\dot{u}}_{L^2}
            +\norm{\sqrt{\rho}(\theta-\tilde{\theta})}_{L^2}+\norm{\nabla\theta
            }_{L^2}+\norm{\rho-\tilde{\rho}}_{L^2}+1.
          \end{aligned}
          \end{equation}
        \end{enumerate}
        \noindent Thus, by substituting \eqref{theta-L^2-estimate-blowup-u3-2-1}-\eqref{theta-L^2-estimate-blowup-u3-2-3} into \eqref{theta-L^2-estimate-blowup-u3-2}, we get
        \begin{equation}\label{theta-L^2-estimate-blowup-u3-2-F}
          \begin{aligned}
          &\quad I_2=\int|u|^2|\nabla u|^2\\
          &\leq C_{\epsilon_1}\left(1+\norm{u}_{L^{r,\infty}}^s\right)\norm{\nabla u}_{L^2}^2+\epsilon_1\left(\norm{\nabla\theta}_{L^2}^2+\norm{\sqrt{\rho}\dot{u}}_{L^2}^2\right)\\
          &\qquad+C\left(\norm{\sqrt{\rho}(\theta-\tilde{\theta})}_{L^2}^2+\norm{\rho-\tilde{\rho}}_{L^2}^2+1\right).
          \end{aligned}
        \end{equation}
  \end{itemize}
  Combining the above estimates, we deduce \eqref{theta-L^2-estimate-blowup-u}.
  Therefore, the proof of Lemma \ref{lem-theta-L^2-blowup-u} is completed.
\end{proof}
\vspace{-0.2cm}
Next, we derive the estimate on $\displaystyle\int\rho|u|^2+\int|\rho-\tilde{\rho}|^2$.
\begin{lem}\label{lem-u-rho-L^2-estimate-u}
  Under the assumptions of Theorem \ref{thm-blowup-FCNS-velocity} and \eqref{assumption2}, it holds that for some constant $C_1>0$,
  \begin{equation}\label{u-rho-L^2-estimate-u}
    \frac{d}{dt}\int(\rho|u|^2+|\rho-\tilde{\rho}|^2)+C_1\mu\int|\nabla u|^2\leq C(\int\rho|\theta-\tilde{\theta}|^2+\int|\rho-\tilde{\rho}|^2).
  \end{equation}
\end{lem}

\begin{proof}
   Multiplying $\eqref{FCNS-eq}_2$ by $u$ and then integrating the resultant equation over $\Omega$, we have from \eqref{assumption2} that
  \begin{equation}\label{u-rho-L^2-estimate-u1}
    \begin{aligned}
    &\quad\frac{d}{dt}\int\rho|u|^2+2\int\mu|\curl u|^2+(2\mu+\lambda)|\div u|^2\\
    &=2\int R(\rho\theta-\tilde{\rho}\tilde{\theta})\div u\\
    &\leq\epsilon\int|\nabla u|^2+C_{\epsilon}\left(\int\rho^2|\theta-\tilde{\theta}|^2+\int|\rho-\tilde{\rho}|^2\right)\\
    &\leq\epsilon\int|\nabla u|^2+C_{\epsilon}\left(\int\rho|\theta-\tilde{\theta}|^2+\int|\rho-\tilde{\rho}|^2\right).
    \end{aligned}
  \end{equation}
   Here, it should be mentioned that for bounded domain $\Omega$ with Navier-slip boundary condition \eqref{Navier-slip-condition}, the original term $\displaystyle\int\mu|\nabla u|^2+(\mu+\lambda)|\div u|^2$ in \eqref{u-rho-L^2-estimate-u1} has been replaced by $\displaystyle\int\mu|\curl u|^2+(2\mu+\lambda)|\div u|^2$, which can also control $\displaystyle\int|\nabla u|^2$ by \eqref{Hodge-decomposition2}. For convenience and consistence, we directly use $\displaystyle\int\mu|\curl u|^2+(2\mu+\lambda)|\div u|^2$ in this lemma.

  Recalling that
  \begin{equation*}
    [(\rho-\tilde{\rho})^2]_t+\div[(\rho-\tilde{\rho})^2u]+(\rho-\tilde{\rho})^2\div u+2\tilde{\rho}(\rho-\tilde{\rho})\div u=0,
  \end{equation*}
  which together with \eqref{assumption2} implies
  \begin{equation}\label{u-rho-L^2-estimate-u2}
    \begin{aligned}
    \frac{d}{dt}\int|\rho-\tilde{\rho}|^2&\leq\epsilon\norm{\nabla u}_{L^2}^2+C_{\epsilon}\left(\norm{\rho-\tilde{\rho}}_{L^4}^4+\norm{\rho-\tilde{\rho}}_{L^2}^2\right)\\
    &\leq\epsilon\norm{\nabla u}_{L^2}^2+C_{\epsilon}\norm{\rho-\tilde{\rho}}_{L^2}^2.
    \end{aligned}
  \end{equation}
  Then, combining \eqref{u-rho-L^2-estimate-u1} and \eqref{u-rho-L^2-estimate-u2} leads to \eqref{u-rho-L^2-estimate-u}, and thus completes the proof of Lemma \ref{lem-u-rho-L^2-estimate-u}.
\end{proof}
Next, we dedcuce the estimate on $\norm{\nabla u}_{L^{\infty}(0,T;L^2)}$.
\begin{lem}\label{lem-u-1rd-estimate-u}
  Under the assumptions of Theorem \ref{thm-blowup-FCNS-velocity} and \eqref{assumption2}, it holds that
  \begin{itemize}
    \item If $\Omega=\mathbb{R}^3$, then
    \begin{equation}\label{u-1rd-estimate-u-Cauchy}
      \begin{aligned}
      &\quad\frac{d}{dt}\int[\mu|\nabla u|^2+(\mu+\lambda)|\div u|^2-2R(\rho\theta-\tilde{\rho}\tilde{\theta})\div u+\frac{1}{2\mu+\lambda}|R(\rho\theta-\tilde{\rho}\tilde{\theta})|^2]+\int\rho|\dot{u}|^2\\
      &\leq C\left(1+\norm{u}_{L^{r,\infty}}^s\right)\left(\norm{\nabla u}_{L^2}^2+\norm{\sqrt{\rho}u}_{L^2}^2+\norm{\sqrt{\rho}(\theta-\tilde{\theta})}_{L^2}^2+\norm{\rho-\tilde{\rho}}_{L^2}^2+1\right)
      +C\norm{\nabla\theta}_{L^2}^2.
      \end{aligned}
    \end{equation}
    \item If $\Omega$ is bounded with Dirichlet boundary condition \eqref{Dirichlet-condition}, then
    \begin{equation}\label{u-1rd-estimate-u-Dirichlet}
      \begin{aligned}
      &\quad\frac{d}{dt}\int[\mu|(\nabla u,\nabla h)|^2+(\mu+\lambda)|(\div u,\div h)|^2-2R(\rho\theta-\tilde{\rho}\tilde{\theta})\div u+\frac{|R(\rho\theta-\tilde{\rho}\tilde{\theta})|^2}{2\mu+\lambda}]
      +\int\rho|\dot{u}|^2\\
      &\leq C\left(1+\norm{u}_{L^{r,\infty}}^s\right)\left(\norm{\nabla u}_{L^2}^2+\norm{\sqrt{\rho}u}_{L^2}^2+\norm{\sqrt{\rho}(\theta-\tilde{\theta})}_{L^2}^2+\norm{\rho-\tilde{\rho}}_{L^2}^2+1\right)
      +C\norm{\nabla\theta}_{L^2}^2.
      \end{aligned}
    \end{equation}
    \item If $\Omega$ is bounded with Navier-slip boundary condition \eqref{Navier-slip-condition}, then
    \begin{equation}\label{u-1rd-estimate-u-Navier}
      \begin{aligned}
      &\quad\frac{d}{dt}\int[\mu|\curl u|^2+(2\mu+\lambda)|\div u|^2-2R(\rho\theta-\tilde{\rho}\tilde{\theta})\div u+\frac{1}{2\mu+\lambda}|R(\rho\theta-\tilde{\rho}\tilde{\theta})|^2]+\int\rho|\dot{u}|^2\\
      &\leq C\left(1+\norm{u}_{L^{r,\infty}}^s\right)\left(\norm{\nabla u}_{L^2}^2+\norm{\sqrt{\rho}u}_{L^2}^2+\norm{\sqrt{\rho}(\theta-\tilde{\theta})}_{L^2}^2+\norm{\rho-\tilde{\rho}}_{L^2}^2+1\right)
      +C\norm{\nabla\theta}_{L^2}^2.
      \end{aligned}
    \end{equation}
  \end{itemize}
\end{lem}

\begin{proof}
  Multiplying $\eqref{FCNS-eq}_2$ by $u_t$, and then integrating the resultant equation over $\Omega$, we have from integrating by parts that
  \begin{equation}\label{u-1rd-estimate-u1}
    \begin{aligned}
    &\quad\frac{1}{2}\frac{d}{dt}\int(\mu|\curl u|^2+(2\mu+\lambda)|\div u|^2)+\int\rho|\dot{u}|^2\\
    &=-\int\nabla P\cdot u_t-\int\rho u\cdot\nabla u\cdot\dot{u}\\
    &\leq\frac{d}{dt}\int R(\rho\theta-\tilde{\rho}\tilde{\theta})\div u-\int P_t\div u+\frac{1}{4}\int\rho|\dot{u}|^2+C\int\rho|u|^2|\nabla u|^2\\
    &\leq\frac{d}{dt}\int R(\rho\theta-\tilde{\rho}\tilde{\theta})\div u-\frac{1}{2(2\mu+\lambda)}\frac{d}{dt}\int|R(\rho\theta-\tilde{\rho}\tilde{\theta})|^2-\frac{1}{2\mu+\lambda}\int P_tG\\
    &\qquad+\frac{1}{4}\int\rho|\dot{u}|^2+C\int|u|^2|\nabla u|^2.
    \end{aligned}
  \end{equation}
   It should be mentioned that for $\Omega=\mathbb{R}^3$ or bounded domain $\Omega$ with Dirichlet boundary condition \eqref{Dirichlet-condition}, $\displaystyle\int[\mu|\curl u|^2+(2\mu+\lambda)|\div u|^2]$\footnote{It can provide the energy $\displaystyle\int|\nabla u|^2$ by \eqref{Hodge-decomposition1} and \eqref{Hodge-decomposition2}.} can be replaced by $\displaystyle\int[\mu|\nabla u|^2+(\mu+\lambda)|\div u|^2]$. We take the form in \eqref{u-1rd-estimate-u1} for consistence. It is easy to check that the energy in \eqref{u-1rd-estimate-u1}
  \begin{equation*}
  \int[\mu|\curl u|^2+(2\mu+\lambda)|\div u|^2-2R(\rho\theta-\tilde{\rho}\tilde{\theta})\div u+\frac{1}{2\mu+\lambda}|R(\rho\theta-\tilde{\rho}\tilde{\theta})|^2]\geq\int\mu|\curl u|^2,
  \end{equation*}
  which however is not enough to close the 1-order energy estimate since there is no control on the energy $\displaystyle\int|\div u|^2$. On the other hand, noticing that  $\rho$ is bounded, we have
  \begin{equation*}
    \int-2R(\rho\theta-\tilde{\rho}\tilde{\theta})\div u\leq\frac{\mu}{2}\int|\div u|^2+C\int\rho|\theta-\tilde{\theta}|^2+C\int|\rho-\tilde{\rho}|^2.
  \end{equation*}
  Then, we can combine Lemma \ref{lem-theta-L^2-blowup-u} and Lemma \ref{lem-u-rho-L^2-estimate-u} to close the energy estimate.

   For $\Omega=\mathbb{R}^3$, due to \eqref{u-1rd-estimate-Cauchy3} and \eqref{GW-estimate1}, we have
  \begin{equation}\label{u-1rd-estimate-u2-1}
    \begin{aligned}
    &\quad-\frac{1}{2\mu+\lambda}\int P_tG\\
    &\leq\epsilon\norm{\sqrt{\rho}\dot{u}}_{L^2}^2+\epsilon\norm{\nabla G}_{L^2}^2+C_{\epsilon}\left(\int\rho^2|\theta-\tilde{\theta}|^2|u|^2+\int\rho^2|u|^2+\norm{\nabla\theta}_{L^2}^2+\int\rho|u|^2|G|^2\right)\\
    &\leq C_{\epsilon}\left(\int(\rho^2+\rho^3)|\theta-\tilde{\theta}|^2|u|^2+\int\rho^2|u|^2+\norm{\nabla\theta}_{L^2}^2+\int(\rho+\rho^3)|u|^2+\int\rho|u|^2|\nabla u|^2\right)\\
    &\qquad+C\epsilon\norm{\sqrt{\rho}\dot{u}}_{L^2}^2\\
    &\leq C_{\epsilon}\left(\int\rho|\theta-\tilde{\theta}|^2|u|^2+\int\rho|u|^2+\norm{\nabla\theta}_{L^2}^2+\int|u|^2|\nabla u|^2\right)+C\epsilon\norm{\sqrt{\rho}\dot{u}}_{L^2}^2\\
    &\leq C_{\epsilon}\left(1+\norm{u}_{L^{r,\infty}}^s\right)\left(\norm{\nabla u}_{L^2}^2+\norm{\sqrt{\rho}u}_{L^2}^2+\norm{\sqrt{\rho}(\theta-\tilde{\theta})}_{L^2}^2+\norm{\rho-\tilde{\rho}}_{L^2}^2+1\right)\\
    &\qquad+C\epsilon\norm{\sqrt{\rho}\dot{u}}_{L^2}^2+C_{\epsilon}\norm{\nabla\theta}_{L^2}^2,
    \end{aligned}
  \end{equation}
  where we have used the estimates on $\displaystyle\int\rho|\theta-\tilde{\theta}|^2|u|^2$ and $\displaystyle\int|u|^2|\nabla u|^2$ in Lemma \ref{lem-theta-L^2-blowup-u}.

  For the bounded domain $\Omega$ with Navier-slip boundary condition \eqref{Navier-slip-condition}, there exist additional boundary terms as \eqref{u-1rd-boundary} and \eqref{u-1rd-estimate-boundary}:
  \begin{equation}\label{boundary-term2}
    \begin{aligned}
    &\quad-\int_{\pa\Omega}2\mu\mathcal{D}(u)n\cdot uG\\
    &=-\int_{\pa\Omega}2\mu(\mathcal{D}(u)n)_{\tau}\cdot u_{\tau}G=\int_{\pa\Omega}2\mu\kappa_{\tau}|u_{\tau}|^2G\\
    &\leq C\norm{G}_{L^2(\pa\Omega)}\norm{|u|^2}_{L^2(\pa\Omega)}
    \lesssim\norm{G}_{L^2}^{\frac{1}{2}}\norm{G}_{H^1}^{\frac{1}{2}}\norm{|u|^2}_{L^2}^{\frac{1}{2}}\norm{|u|^2}_{H^1}^{\frac{1}{2}}\\
    &\leq\epsilon\norm{\nabla G}_{L^2}^2+\epsilon\norm{u|\nabla u|}_{L^2}^2+C_{\epsilon}(\norm{G}_{L^2}^2+\norm{|u|^2}_{L^2}^2)\\
    &\leq C_{\epsilon}\left(1+\norm{u}_{L^{r,\infty}}^s\right)\left(\norm{\nabla u}_{L^2}^2+\norm{\sqrt{\rho}u}_{L^2}^2+\norm{\sqrt{\rho}(\theta-\tilde{\theta})}_{L^2}^2+\norm{\rho-\tilde{\rho}}_{L^2}^2+1\right)\\
    &\qquad+C\epsilon\norm{\sqrt{\rho}\dot{u}}_{L^2}^2+\epsilon\norm{\nabla\theta}_{L^2}^2,
    \end{aligned}
  \end{equation}
  where we have used the estimates on $\displaystyle\int|u|^4$ and $\displaystyle\int|u|^2|\nabla u|^2$ in Lemma \ref{lem-theta-L^2-blowup-u}. Thus for the bounded domain $\Omega$ with Navier-slip boundary condition \eqref{Navier-slip-condition}, we can get the same estimate as \eqref{u-1rd-estimate-u2-1}.

  For the bounded domain $\Omega$ with Dirichlet boundary condition \eqref{Dirichlet-condition}, due to the absence of estimate on $\norm{\nabla G}_{L^2}$, we deal with the term $\displaystyle-\int P_t\div u$ in a different way. To see this, by decomposition $u=h+g$, we have from \eqref{u-1rd-estimate2} and \eqref{u-1rd-estimate4} that
  \begin{equation}\label{u-1rd-estimate-u2-2}
    \begin{aligned}
    -\int P_t\div u&=-\frac{1}{2}\frac{d}{dt}\int(\mu|\nabla h|^2+(\mu+\lambda)|\div h|^2)-\int P_t\div g\\
    &\leq-\frac{1}{2}\frac{d}{dt}\int(\mu|\nabla h|^2+(\mu+\lambda)|\div h|^2)+C\epsilon\norm{\sqrt{\rho}\dot{u}}_{L^2}^2\\
    &\quad+C_{\epsilon}\left(\norm{u|\nabla u|}_{L^2}^2+\int\rho|\theta-\tilde{\theta}|^2|u|^2+\norm{\nabla u}_{L^2}^2+\int|u|^2|\div h|^2+\norm{\nabla\theta}_{L^2}^2\right)\\
    &\leq-\frac{1}{2}\frac{d}{dt}\int(\mu|\nabla h|^2+(\mu+\lambda)|\div h|^2)+C\epsilon\norm{\sqrt{\rho}\dot{u}}_{L^2}^2+C_{\epsilon}\norm{\nabla\theta}_{L^2}^2\\
    &\quad+C_{\epsilon}(1+\norm{u}_{L^{r,\infty}}^s)\left(\norm{\nabla u}_{L^2}^2+\norm{\sqrt{\rho}u}_{L^2}^2+\norm{\sqrt{\rho}(\theta-\tilde{\theta})}_{L^2}^2+\norm{\rho-\tilde{\rho}}_{L^2}^2+1\right),
    \end{aligned}
  \end{equation}
  where we have used the estimates on $\displaystyle\int\rho|u|^2|\theta-\tilde{\theta}|^2$ and $\displaystyle\int|u|^2|\nabla u|^2$ in Lemma \ref{lem-theta-L^2-blowup-u}.\\
  Combining the above relations, we conclude that
  \begin{itemize}
    \item If $\Omega=\mathbb{R}^3$, then
    \begin{equation}\label{u-1rd-estimate-u3-1}
      \begin{aligned}
      &\quad\frac{d}{dt}\int[\mu|\nabla u|^2+(\mu+\lambda)|\div u|^2-2R(\rho\theta-\tilde{\rho}\tilde{\theta})\div u+\frac{1}{2\mu+\lambda}|R(\rho\theta-\tilde{\rho}\tilde{\theta})|^2]+\int\rho|\dot{u}|^2\\
      &\leq C(1+\norm{u}_{L^{r,\infty}}^s)\left(\norm{\nabla u}_{L^2}^2+\norm{\sqrt{\rho}u}_{L^2}^2+\norm{\sqrt{\rho}(\theta-\tilde{\theta})}_{L^2}^2+\norm{\rho-\tilde{\rho}}_{L^2}^2+1\right)
      +C\norm{\nabla\theta}_{L^2}^2.
      \end{aligned}
    \end{equation}
    \item If $\Omega$ is bounded with Dirichlet boundary condition \eqref{Dirichlet-condition}, then
    \begin{equation}\label{u-1rd-estimate-u3-2}
      \begin{aligned}
      &\quad\frac{d}{dt}\int[\mu|(\nabla u,\nabla h)|^2+(\mu+\lambda)|(\div u,\div h)|^2-2R(\rho\theta-\tilde{\rho}\tilde{\theta})\div u+\frac{|R(\rho\theta-\tilde{\rho}\tilde{\theta})|^2}{2\mu+\lambda}]
      +\int\rho|\dot{u}|^2\\
      &\leq C(1+\norm{u}_{L^{r,\infty}}^s)\left(\norm{\nabla u}_{L^2}^2+\norm{\sqrt{\rho}u}_{L^2}^2+\norm{\sqrt{\rho}(\theta-\tilde{\theta})}_{L^2}^2+\norm{\rho-\tilde{\rho}}_{L^2}^2+1\right)
      +C\norm{\nabla\theta}_{L^2}^2.
      \end{aligned}
    \end{equation}
    \item If $\Omega$ is bounded with Navier-slip boundary condition \eqref{Navier-slip-condition}, then
    \begin{equation}\label{u-1rd-estimate-u3-3}
      \begin{aligned}
      &\quad\frac{d}{dt}\int[\mu|\curl u|^2+(2\mu+\lambda)|\div u|^2-2R(\rho\theta-\tilde{\rho}\tilde{\theta})\div u+\frac{1}{2\mu+\lambda}|R(\rho\theta-\tilde{\rho}\tilde{\theta})|^2]+\int\rho|\dot{u}|^2\\
      &\leq C(1+\norm{u}_{L^{r,\infty}}^s)\left(\norm{\nabla u}_{L^2}^2+\norm{\sqrt{\rho}u}_{L^2}^2+\norm{\sqrt{\rho}(\theta-\tilde{\theta})}_{L^2}^2+\norm{\rho-\tilde{\rho}}_{L^2}^2+1\right)
      +C\norm{\nabla\theta}_{L^2}^2.
      \end{aligned}
    \end{equation}
  \end{itemize}
  Therefore, we have completed the proof of Lemma \ref{lem-u-1rd-estimate-u}.

\end{proof}

Combining Lemma \ref{lem-theta-L^2-blowup-u}, Lemma \ref{lem-u-rho-L^2-estimate-u} with Lemma \ref{lem-u-1rd-estimate-u}, it is easy to get the following lemma.

\begin{lem}\label{lem-u-1rd-estimate-u-close}
  Under the assumptions of Theorem \ref{thm-blowup-FCNS-velocity} and \eqref{assumption2}, it holds that
  \begin{equation}\label{u-1rd-estimate-u-close}
    \sup_{t\in[0,T]}\int(|\nabla u|^2+\rho|u|^2+|\rho-\tilde{\rho}|^2+\rho|\theta-\tilde{\theta}|^2)+\int_0^T\int(|\nabla u|^2+\rho|\dot{u}|^2)\leq C,
  \end{equation}
  for any $T\in(0,T^*)$.
\end{lem}

{\bf{Proof of Theorem Theorem \ref{thm-blowup-FCNS-velocity}.}}
Combining Lemma \ref{lem-theta-L^2-blowup-u}-Lemma \ref{lem-u-1rd-estimate-u-close} and modifying the methods of \cite{Huang-Li-Wang2013} and \cite{Feireisl2024}, we conclude that
\begin{equation}\label{u-2rd-estimate-u}
  \sup_{t\in[0,T^*)}(\norm{\rho-\tilde{\rho}}_{H^1\cap W^{1,q}}+\norm{\nabla u}_{H^1}+\norm{\nabla\theta}_{H^1})\leq C,
\end{equation}
 for three different cases. It is worth mentioning that for bounded domain $\Omega$ with Navier-slip boundary condition \eqref{Navier-slip-condition}, we need to
  tackle with additional boundary terms similar as \eqref{boundary-term1} and \eqref{boundary-term2} carefully.
  For the sake of simplicity, we omit the details here.
  Combining the available local existence results stated in Remark \ref{rem-local-existence} with \eqref{u-2rd-estimate-u}, it is clear that  the strong solution can be extended beyond $T^*$ (see \cite{Huang-Li-Wang2013} for instance), which means $T^*$ is not the maximal existence time, and thus leads to a contradiction.
Therefore, we complete the proof of Theorem \ref{thm-blowup-FCNS-velocity}.

\section{Proof of Theorem \ref{thm-blowup-CNS}}
Assume that $T^*<\infty$ and there exist constants $r\in(3,\infty]$, $s\in[2,\infty]$ and sufficiently large $\alpha>0$ satisfying \eqref{index-blowup3} such that
\begin{equation}\label{assumption3}
  \norm{\rho}_{L^{\infty}(0,T;L^{\alpha})}+\norm{u}_{L^s(0,T;L^r_w)}\leq M^*<\infty,
\end{equation}
for any $T\in(0,T^*)$. Our aim is to show that
\begin{equation}\label{goal-inequality2}
  \sup_{t\in[0,T]}(\norm{\rho-\tilde{\rho}}_{H^1\cap W^{1,q}}+\norm{\nabla u}_{H^1})\leq C.
\end{equation}
Then a contradiction argument as in Section \ref{sect-proof-theorem-FCNS-theta} will prove Theorem \ref{thm-blowup-CNS}. It is worth mentioning that due to the absence of the technical condition $\rho_0\in L^1$, the estimate in \cite{Wang2020} is not applicable in our case and thus some new strategies should be introduced to overcome this difficulty. Moreover, if $\alpha<\infty$, $\tilde{\rho}=0$ should hold. For convenience, we take $\tilde{\rho}=0$ in this whole section.

Next, we first give the key estimate on $\norm{\nabla u}_{L^{\infty}(0,T;L^2)}$.

\begin{lem}\label{lem-u-1rd-CNS-u}
   Under the assumptions of Theorem \ref{thm-blowup-CNS} and \eqref{assumption3}, it holds that
   \begin{equation}\label{u-1rd-CNS-u}
     \sup_{t\in[0,T]}\int(|\nabla u|^2+|\rho|^2)+\int_0^T\norm{\sqrt{\rho}\dot{u}}_{L^2}^2\leq C,
   \end{equation}
   for any $T\in(0,T^*)$.
\end{lem}

\begin{proof}
  Recalling $\eqref{CNS-eq}_1$, it is easy to check that
  \begin{equation*}
    \pa_t(\rho^2)+\div(\rho^2u)+\rho^2\div u=0.
  \end{equation*}
  Then
  \begin{equation}\label{rho-L^2-CNS}
    \frac{d}{dt}\int|\rho|^2\leq\int|\div u|^2+\int\rho^4\leq\int|\div u|^2+C(\int|\rho|^2+1),
  \end{equation}
  where we have used the following arguments as
  \begin{itemize}
    \item if $\alpha=\infty$, by \eqref{assumption3}, we have
    \begin{equation}\label{rho-L^4-CNS1}
      \int\rho^4\leq C\int\rho^2,
    \end{equation}
    \item if $\alpha<\infty$, by \eqref{assumption3} and $\alpha>4$, we have
    \begin{equation}\label{rho-L^4-CNS2-1}
    \begin{aligned}
      \int\rho^4&\leq\norm{\rho}_{L^{\alpha}}^2\norm{\rho}_{L^{\frac{2\alpha}{\alpha-2}}}^2
      \leq C\norm{\rho}_{L^2}^{2(1-\frac{4}{\alpha})}\norm{\rho}_{L^4}^{\frac{8}{\alpha}}
      \leq \frac{1}{2}\norm{\rho}_{L^4}^4+C(\norm{\rho}_{L^2}^2+1),
    \end{aligned}
    \end{equation}
    which implies
    \begin{equation}\label{rho-L^4-CNS2}
      \int\rho^4\leq C(\int\rho^2+1).
    \end{equation}
  \end{itemize}
  Indeed, we can apply similar arguments to get
  \begin{equation}\label{rho-L^p-CNS}
      \int\rho^p\leq C(\int\rho^2+1),
    \end{equation}
    for any $p\in(2,\alpha)$.
  On the other hand, multiplying $\eqref{CNS-eq}_2$ by $u_t$ and then integrating the resultant equation over $\Omega=\mathbb{R}^3$, we obtain
  \begin{equation}\label{u-1rd-CNS-u1}
    \begin{aligned}
    &\quad\frac{1}{2}\frac{d}{dt}\int(\mu|\nabla u|^2+(\mu+\lambda)|\div u|^2)+\int\rho|\dot{u}|^2\\
    &=\int P\div u_t+\int\rho u\cdot\nabla u\cdot\dot{u}\\
    &\leq\frac{d}{dt}\int P\div u-\int P_t\div u+\frac{1}{4}\int\rho|\dot{u}|^2+C\int\rho|u|^2|\nabla u|^2\\
    &\leq\frac{d}{dt}\int P\div u-\int Pu\cdot\nabla\div u+\int(\gamma-1)P(\div u)^2\\
    &\qquad+\frac{1}{4}\int\rho|\dot{u}|^2+C\int\rho|u|^2|\nabla u|^2\\
    &\leq\frac{d}{dt}\int P\div u-\frac{1}{2\mu+\lambda}\int Pu\cdot\nabla G+\frac{1}{2(2\mu+\lambda)}\int P^2\div u\\
    &\qquad+(\gamma-1)\int P|\div u|^2+\frac{1}{4}\int\rho|\dot{u}|^2+C\int\rho|u|^2|\nabla u|^2,
    \end{aligned}
  \end{equation}
  where we have used the fact that
  \begin{equation*}
    P_t+\div(Pu)+(\gamma-1)P\div u=0,
  \end{equation*}
  and the effective viscous flux $G=(2\mu+\lambda)\div u-P$. Then
  \begin{itemize}
    \item if $\alpha=\infty$, by \eqref{assumption3} and \eqref{GW-estimate1}, we have
        \begin{equation}\label{u-1rd-CNS-u2-1}
          \begin{aligned}
          K_1&:=-\frac{1}{2\mu+\lambda}\int Pu\cdot\nabla G+\frac{1}{2(2\mu+\lambda)}\int P^2\div u
          +(\gamma-1)\int P|\div u|^2\\
          &\leq\epsilon\norm{\nabla G}_{L^2}^2+C_{\epsilon}\norm{P u}_{L^2}^2+C\left(\norm{\nabla u}_{L^2}^2+\norm{P^2}_{L^2}^2\right)+C\int P|\div u|^2\\
          &\leq C\epsilon\norm{\sqrt{\rho}\dot{u}}_{L^2}^2+C_{\epsilon}\norm{\rho u}_{L^2}^2+C\left(\norm{\nabla u}_{L^2}^2+\norm{\rho}_{L^2}^2\right)\\
          &\leq C\epsilon\norm{\sqrt{\rho}\dot{u}}_{L^2}^2+C_{\epsilon}\left(1+\norm{u}_{L^{r,\infty}}^s\right)\left(\norm{\rho}_{L^2}^2+1)+C(\norm{\nabla u}_{L^2}^2+\norm{\rho}_{L^2}^2\right),
          \end{aligned}
        \end{equation}
        where we have used  \eqref{Lorentz-Holder}, \eqref{Lorentz-interpolation} and \eqref{rho-L^p-CNS} to get
        \begin{equation*}
          \begin{aligned}
          \norm{\rho|u|}_{L^2}^2\lesssim\norm{u}_{L^{r,\infty}}^2\norm{\rho}_{L^{\frac{2r}{r-2},2}}^2
          \lesssim\norm{u}_{L^{r,\infty}}^2\norm{\rho}_{L^{\frac{2r_1}{r_1-2}}}\norm{\rho}_{L^{\frac{2r_2}{r_2-2}}}\lesssim
          \left(1+\norm{u}_{L^{r,\infty}}^s\right)\left(\norm{\rho}_{L^2}^2+1\right),
          \end{aligned}
        \end{equation*}
        where $r_1=r_2=r=\infty$ or $3<r_1<r<r_2<\infty$ with $\frac{2}{r}=\frac{1}{r_1}+\frac{1}{r_2}$.
    \item if $\alpha<\infty$, we have
        \begin{equation}\label{u-1rd-CNS-u2-2}
          \begin{aligned}
          K_1&=-\frac{1}{2\mu+\lambda}\int Pu\cdot\nabla G+\frac{1}{2(2\mu+\lambda)}\int P^2\div u
          +(\gamma-1)\int P|\div u|^2\\
          &\leq\epsilon\norm{\nabla G}_{L^{\frac{1}{\frac{1}{2}+\frac{1}{2\alpha}}}}^2+C_{\epsilon}\norm{P u}_{L^{\frac{1}{\frac{1}{2}-\frac{1}{2\alpha}}}}^2+C(\norm{\nabla u}_{L^2}^2+\norm{P^2}_{L^2}^2)+C\norm{\rho}_{L^{\alpha}}^{\gamma}\norm{\div u}_{L^{\frac{2}{1-\frac{\gamma}{\alpha}}}}^2\\
          &\leq C\epsilon\norm{\rho}_{L^{\alpha}}\norm{\sqrt{\rho}\dot{u}}_{L^2}^2+C_{\epsilon}\norm{ u}_{L^{r,\infty}}^2\norm{P}_{L^{\frac{1}{\frac{1}{2}-\frac{1}{2\alpha}-\frac{1}{r}},2}}^2+C\left(\norm{\nabla u}_{L^2}^2+\norm{\rho}_{L^2}^2+1\right)\\
          &\quad+C\norm{G+P}_{L^{\frac{2}{1-\frac{\gamma}{\alpha}}}}^2\\
          &\leq C\epsilon\norm{\sqrt{\rho}\dot{u}}_{L^2}^2+C_{\epsilon}\left(1+\norm{u}_{L^{r,\infty}}^s\right)\left(\norm{\rho}_{L^2}^2+1\right)+C\left(\norm{\nabla u}_{L^2}^2+\norm{\rho}_{L^2}^2+1\right),
          \end{aligned}
        \end{equation}
        where we have used \eqref{Lorentz-interpolation} and \eqref{rho-L^p-CNS} to get
        \begin{equation*}
          \norm{P}_{L^{\frac{1}{\frac{1}{2}-\frac{1}{2\alpha}-\frac{1}{r}},2}}^2\leq\norm{P}_{L^{\frac{1}{\frac{1}{2}-
          \frac{1}{2\alpha}-\frac{1}{r_1}}}}\norm{P}_{L^{\frac{1}{\frac{1}{2}-\frac{1}{2\alpha}-\frac{1}{r_2}}}}\leq C\left(\norm{\rho}_{L^2}^2+1\right),
        \end{equation*}
        and used Gagliardo-Nirenberg interpolation inequality and \eqref{GW-estimate1} to obtain
        \begin{equation*}
        \begin{aligned}
          \norm{G}_{L^{\frac{2}{1-\frac{\gamma}{\alpha}}}}^2&\leq C_{\epsilon}\norm{G}_{L^2}^2+\epsilon\norm{\nabla G}_{L^{\frac{1}{\frac{1}{2}+\frac{1}{2\alpha}}}}^2
          \leq C_{\epsilon}\left(\norm{\nabla u}_{L^2}^2+\norm{P}_{L^2}^2\right)+C\epsilon\norm{\rho}_{L^{\alpha}}\norm{\sqrt{\rho}\dot{u}}_{L^2}^2\\
          &\leq C_{\epsilon}\left(\norm{\nabla u}_{L^2}^2+\norm{\rho}_{L^2}^2+1\right)+C\epsilon\norm{\sqrt{\rho}\dot{u}}_{L^2}^2.
        \end{aligned}
        \end{equation*}
        Here $\alpha<\infty$ is sufficiently large such that
        \begin{equation*}
          \alpha>2\gamma, \quad\frac{\gamma}{\frac{1}{2}-\frac{1}{2\alpha}-\frac{1}{r}}<\alpha,\quad\frac{2}{1-\frac{\gamma}{\alpha}}<\frac{1}{\frac{1}{2}+\frac{1}{2\alpha}-\frac{1}{3}}=\frac{1}{\frac{1}{6}+\frac{1}{2\alpha}},\quad\frac{2}{1-\frac{\gamma}{\alpha}}<\alpha.
        \end{equation*}
  \end{itemize}
   Now, we turn to estimate $\displaystyle K_2:=\int\rho|u|^2|\nabla u|^2$. We divide our proofs into two cases:
   \begin{itemize}
     \item If $\alpha=\infty$, by \eqref{Lorentz-Holder}, \eqref{Lorentz-interpolation}, \eqref{Hodge-decomposition1}, \eqref{GW-estimate1} and \eqref{rho-L^p-CNS}, we have
     \begin{equation}\label{u-1rd-CNS-u3-1}
       \begin{aligned}
       K_2&=\int\rho|u|^2|\nabla u|^2\leq C\norm{u}_{L^{r,\infty}}^2\norm{\nabla u}_{L^{\frac{2r}{r-2},2}}^2\\
       &\leq C\norm{u}_{L^{r,\infty}}^2\norm{\nabla u}_{L^{\frac{2r_1}{r_1-2}}}\norm{\nabla u}_{L^{\frac{2r_2}{r_2-2}}}\\
       &\leq C\norm{u}_{L^{r,\infty}}^2\norm{\nabla u}_{L^2}^{1-\frac{3}{r_1}}\norm{\nabla u}_{L^6}^{\frac{3}{r_1}}\norm{\nabla u}_{L^2}^{1-\frac{3}{r_2}}\norm{\nabla u}_{L^6}^{\frac{3}{r_2}}\\
       &\leq C\norm{u}_{L^{r,\infty}}^2\norm{\nabla u}_{L^2}^{2-\frac{6}{r}}\norm{\nabla u}_{L^6}^{\frac{6}{r}}\\
       &\leq C_{\epsilon}\norm{u}_{L^{r,\infty}}^{\frac{2}{1-\frac{3}{r}}}\norm{\nabla u}_{L^2}^2+\epsilon\frac{1}{r}\norm{\nabla u}_{L^6}^2\\
       &\leq C_{\epsilon}\left(1+\norm{u}_{L^{r,\infty}}^s\right)\norm{\nabla u}_{L^2}^2+C\epsilon\frac{1}{r}\left(\norm{\nabla G}_{L^2}^2+\norm{\nabla w}_{L^2}^2+\norm{P}_{L^6}^2\right)\\
       &\leq C_{\epsilon}\left(1+\norm{u}_{L^{r,\infty}}^s\right)\norm{\nabla u}_{L^2}^2+C\epsilon\frac{1}{r}\left(\norm{\sqrt{\rho}\dot{u}}_{L^2}^2+\norm{\rho}_{L^2}^2+1\right),
       \end{aligned}
     \end{equation}
      where $r_1=r_2=r=\infty$ or $3<r_1<r<r_2<\infty$ with $\frac{2}{r}=\frac{1}{r_1}+\frac{1}{r_2}$.
     \item If $\alpha<\infty$, we can apply similar arguments to get
     \begin{equation}\label{u-1rd-CNS-u3-2}
       \begin{aligned}
       K_2&=\int\rho|u|^2|\nabla u|^2\leq \norm{\rho}_{L^{\alpha,\infty}}\norm{u}_{L^{r,\infty}}^2\norm{\nabla u}_{L^{\frac{2\bar{r}}{\bar{r}-2},2}}^2\\
       &\leq C\norm{u}_{L^{r,\infty}}^2\norm{\nabla u}_{L^{\frac{2r_1}{r_1-2}}}\norm{\nabla u}_{L^{\frac{2r_2}{r_2-2}}}\\
       &\leq C\norm{u}_{L^{r,\infty}}^2\norm{\nabla u}_{L^2}^{1-\frac{1}{r_1}\frac{1}{\frac{1}{3}-\frac{1}{2\alpha}}}\norm{\nabla u}_{L^{\frac{1}{\frac{1}{6}+\frac{1}{2\alpha}}}}^{\frac{1}{r_1}\frac{1}{\frac{1}{3}-\frac{1}{2\alpha}}}\norm{\nabla u}_{L^2}^{1-\frac{1}{r_2}\frac{1}{\frac{1}{3}-\frac{1}{2\alpha}}}\norm{\nabla u}_{L^{\frac{1}{\frac{1}{6}+\frac{1}{2\alpha}}}}^{\frac{1}{r_2}\frac{1}{\frac{1}{3}-\frac{1}{2\alpha}}}\\
       &\leq C\norm{u}_{L^{r,\infty}}^2\norm{\nabla u}_{L^2}^{2-\frac{2}{\bar{r}}\frac{1}{\frac{1}{3}-\frac{1}{2\alpha}}}\norm{\nabla u}_{L^{\frac{1}{\frac{1}{6}+\frac{1}{2\alpha}}}}^{\frac{2}{\bar{r}}\frac{1}{\frac{1}{3}-\frac{1}{2\alpha}}}\\
       &\leq C_{\epsilon}\norm{u}_{L^{r,\infty}}^{\frac{2}{1-\frac{1}{\bar{r}}\frac{1}{\frac{1}{3}-\frac{1}{2\alpha}}}}\norm{\nabla u}_{L^2}^2+\epsilon\frac{1}{\bar{r}}\norm{\nabla u}_{L^{\frac{1}{\frac{1}{6}+\frac{1}{2\alpha}}}}^2\\
       &\leq C_{\epsilon}(1+\norm{u}_{L^{r,\infty}}^s)\norm{\nabla u}_{L^2}^2+C\epsilon\frac{1}{\bar{r}}\left(\norm{\nabla G}_{L^{\frac{1}{\frac{1}{2}+\frac{1}{2\alpha}}}}^2+\norm{\nabla w}_{L^{\frac{1}{\frac{1}{2}+\frac{1}{2\alpha}}}}^2+\norm{P}_{L^{\frac{1}{\frac{1}{6}+\frac{1}{2\alpha}}}}^2\right)\\
       &\leq C_{\epsilon}\left(1+\norm{u}_{L^{r,\infty}}^s\right)\norm{\nabla u}_{L^2}^2+C\epsilon\frac{1}{\bar{r}}\left(\norm{\rho}_{L^{\alpha}}\norm{\sqrt{\rho}\dot{u}}_{L^2}^2+\norm{\rho}_{L^2}^2+1\right)\\
       &\leq C_{\epsilon}\left(1+\norm{u}_{L^{r,\infty}}^s\right)\norm{\nabla u}_{L^2}^2+C\epsilon\frac{1}{\bar{r}}\left(\norm{\sqrt{\rho}\dot{u}}_{L^2}^2+\norm{\rho}_{L^2}^2+1\right),
       \end{aligned}
     \end{equation}
     where $\frac{1}{\bar{r}}=\frac{1}{r}+\frac{1}{2\alpha}$ and $3<r_1<\bar{r}<r_2<\infty$ with $\frac{2}{\bar{r}}=\frac{1}{r_1}+\frac{1}{r_2}$. Here $\alpha$ is sufficiently large such that
     \begin{equation*}
       \frac{2\bar{r}}{\bar{r}-2}<\frac{1}{\frac{1}{6}+\frac{1}{2\alpha}},\quad\frac{2}{1-\frac{1}{\bar{r}}\frac{1}{\frac{1}{3}-\frac{1}{2\alpha}}}\leq s,\quad\frac{\gamma}{\frac{1}{6}+\frac{1}{2\alpha}}<\alpha.
     \end{equation*}
   \end{itemize}
  \par Finally substituting \eqref{u-1rd-CNS-u2-1}, \eqref{u-1rd-CNS-u2-2}, \eqref{u-1rd-CNS-u3-1} and \eqref{u-1rd-CNS-u3-2} into \eqref{u-1rd-CNS-u1}, and using \eqref{rho-L^2-CNS} and Gronwall inequality, we can prove \eqref{u-1rd-CNS-u}, and thus complete the proof of Lemma \ref{lem-u-1rd-CNS-u}.
\end{proof}
The following lemma concerns the bound of $\norm{\sqrt{\rho}\dot{u}}_{L^{\infty}(0,T;L^2)}$.
\begin{lem}\label{lem-u-2rd-CNS-u}
  Under the assumptions of Theorem \ref{thm-blowup-CNS} and \eqref{assumption3}, it holds that
  \begin{equation}\label{u-2rd-CNS-u}
    \sup_{t\in[0,T]}\norm{\sqrt{\rho}\dot{u}}_{L^2}^2+\int_0^T\norm{\nabla\dot{u}}_{L^2}^2dt\leq C,
  \end{equation}
  for any $T\in(0,T^*)$.
\end{lem}
\begin{proof}
  Due to the compatibility condition \eqref{compatibility-condition}, we can define $\sqrt{\rho}\dot{u}|_{t=0}=g_1\in L^2$. Applying $\dot{u}^j[\pa_t+\div(u\cdot)]$ to the $j$th component of $\eqref{CNS-eq}_2$, and then integrating the resultant equation over $\Omega$, we have from integration by parts that
  \begin{equation}\label{u-2rd-CNS-u1}
  \begin{aligned}
    \frac{d}{dt}\int\frac{1}{2}\rho|\dot{u}|^2&=\int[-\dot{u}^j[\pa_jP_t+\div(\pa_jPu)]+\mu\dot{u}^j[\Delta u_t^j+\div(u\Delta u^j)]\\
    &\quad+(\mu+\lambda)\dot{u}^j[\pa_j\div u_t+\div(u\pa_j\div u)]]\\
    &=\sum_{i=1}^3A_i.
  \end{aligned}
  \end{equation}
  For the term $A_1$, using integration by parts, we have
  \begin{equation}\label{u-2rd-CNS-u2-1}
    \begin{aligned}
      A_1&=\int\div\dot{u}P_t+u\cdot\nabla\dot{u}\cdot\nabla P\\
      &=\int-\gamma\div\dot{u}P\div u-\div\dot{u}u\cdot\nabla P+u\cdot\nabla\dot{u}\cdot\nabla P\\
      &=\int-\gamma\div\dot{u}P\div u+P\div(\div\dot{u}u)-P\div(u\cdot\nabla\dot{u})\\
      &\lesssim\norm{\nabla u}_{L^2}\norm{\nabla\dot{u}}_{L^2}\\
      &\leq\epsilon\norm{\nabla\dot{u}}_{L^2}^2+C_{\epsilon}\norm{\nabla u}_{L^2}^2.
    \end{aligned}
  \end{equation}
  Similarly, we have
  \begin{equation}\label{u-2rd-CNS-u2-2}
    \begin{aligned}
    A_2&=-\mu\int\nabla\dot{u}:\nabla u_t+u\cdot\nabla\dot{u}\cdot\Delta u\\
    &=-\mu\int|\nabla\dot{u}|^2-\nabla\dot{u}:\nabla u\cdot\nabla u-\nabla\dot{u}:(u\cdot\nabla)\nabla u+u\cdot\nabla\dot{u}\cdot\Delta u\\
    &=-\mu\int|\nabla\dot{u}|^2-\nabla\dot{u}:\nabla u\cdot\nabla u+\div u\nabla\dot{u}:\nabla u+u\cdot\nabla(\nabla\dot{u}):\nabla u+u\cdot\nabla\dot{u}\cdot\Delta u\\
    &=-\mu\int|\nabla\dot{u}|^2-\nabla\dot{u}:\nabla u\cdot\nabla u+\div u\nabla\dot{u}:\nabla u-\nabla u\cdot\nabla\dot{u}:\nabla u\\
    &\leq-\frac{3\mu}{4}\norm{\nabla\dot{u}}_{L^2}^2+C\norm{\nabla u}_{L^4}^4,
    \end{aligned}
  \end{equation}
  and
  \begin{equation}\label{u-2rd-CNS-u2-3}
    \begin{aligned}
    A_3&=-(\mu+\lambda)\int\div\dot{u}\div u_t+u\cdot\nabla\dot{u}\cdot\nabla\div u\\
    &=-(\mu+\lambda)\int(\div\dot{u})^2-\div\dot{u}\div(u\cdot\nabla u)+u\cdot\nabla\dot{u}\cdot\nabla\div u\\
    &=-(\mu+\lambda)\int(\div\dot{u})^2-\div\dot{u}\nabla u:\nabla u^t+\div\dot{u}(\div u)^2+u\cdot\nabla\div\dot{u}\div u+u\cdot\nabla\dot{u}\cdot\nabla\div u\\
    &=-(\mu+\lambda)\int(\div\dot{u})^2-\div\dot{u}\nabla u:\nabla u^t+\div\dot{u}(\div u)^2-\div u\nabla u:\nabla\dot{u}^t\\
    &\leq-\frac{\mu+\lambda}{2}\norm{\div\dot{u}}_{L^2}^2+\epsilon\norm{\nabla\dot{u}}_{L^2}^2
    +C_{\epsilon}\norm{\nabla u}_{L^4}^4.
    \end{aligned}
  \end{equation}
  Then, plugging \eqref{u-2rd-CNS-u2-1}-\eqref{u-2rd-CNS-u2-3} into \eqref{u-2rd-CNS-u1} and taking $\epsilon>0$ small enough, we obtain
  \begin{equation}\label{u-2rd-CNS-u3}
  \begin{aligned}
    \frac{d}{dt}\int\rho|\dot{u}|^2+\mu\int|\nabla\dot{u}|^2
    &\leq C\norm{\nabla u}_{L^2}^2+C\norm{\nabla u}_{L^4}^4\\
    &\leq C+C\norm{\nabla u}_{L^2}^{4-\frac{1}{\frac{1}{3}-\frac{1}{2\alpha}}}\norm{\nabla u}_{L^{\frac{1}{\frac{1}{6}+\frac{1}{2\alpha}}}}^{\frac{1}{\frac{1}{3}-\frac{1}{2\alpha}}}\\
    &\leq C+C\left(\norm{\nabla G}_{L^{\frac{1}{\frac{1}{2}+\frac{1}{2\alpha}}}}^4+\norm{\nabla w}_{L^{\frac{1}{\frac{1}{2}+\frac{1}{2\alpha}}}}^4+\norm{P}_{L^{\frac{1}{\frac{1}{6}+\frac{1}{2\alpha}}}}^4\right)\\
    &\leq C+C(\norm{\rho}_{L^{\alpha}}^2\norm{\sqrt{\rho}\dot{u}}_{L^2}^4+\norm{\rho}_{L^2}^4+1)\\
    &\leq C\norm{\sqrt{\rho}\dot{u}}_{L^2}^4+C,
  \end{aligned}
  \end{equation}
 where we have used \eqref{Hodge-decomposition1}, \eqref{GW-estimate1}, \eqref{rho-L^p-CNS}, \eqref{assumption3} and \eqref{u-1rd-CNS-u}. Then, applying Gronwall inequality yields \eqref{u-2rd-CNS-u} immediately. Therefore, the proof of Lemma \ref{lem-u-2rd-CNS-u} is completed.
\end{proof}

Next, we bound the density $\rho$.

\begin{lem}\label{lem-rho-bound}
   Under the assumptions of Theorem \ref{thm-blowup-CNS} and \eqref{assumption3}, it holds that
   \begin{equation}\label{rho-bound}
     \sup_{t\in[0,T]}\norm{\rho}_{L^{\infty}}\leq C,
   \end{equation}
   for any $T\in(0,T^*)$.
\end{lem}

\begin{proof}
  This result is a direct observation from the mass conservation equation $\eqref{CNS-eq}_1$. Indeed, multiplying $\eqref{CNS-eq}_1$ by $p\rho^{p-1}$ with $p>1$ and integrating over $\Omega=\mathbb{R}^3$, we have
  \begin{equation}\label{rho-bound-estimate1}
    \begin{aligned}
    \frac{d}{dt}\int\rho^p&=-(p-1)\int\rho^p\div u=-\frac{p-1}{2\mu+\lambda}\int\rho^pP-\frac{p-1}{2\mu+\lambda}\int\rho^pG\\
    &\leq\frac{p-1}{2\mu+\lambda}\norm{G}_{L^{\infty}}\int\rho^p.
    \end{aligned}
  \end{equation}
  Due to the Gagliardo-Nirenberg inequality, \eqref{u-1rd-CNS-u}, \eqref{rho-L^p-CNS}, \eqref{GW-estimate1} and \eqref{u-2rd-CNS-u}, we have
  \begin{equation}\label{G-bound}
    \begin{aligned}
    \norm{G}_{L^{\infty}}&\leq C\left(\norm{G}_{L^2}+\norm{\nabla G}_{L^{\frac{1}{\frac{1}{6}+\frac{1}{\alpha}}}}\right)\\
    &\leq C\left(1+\norm{\rho}_{L^{\alpha}}\norm{\dot{u}}_{L^6}\right)\\
    &\leq C\left(1+\norm{\nabla\dot{u}}_{L^2}\right).
    \end{aligned}
  \end{equation}
  Then, combining \eqref{G-bound} with \eqref{rho-bound-estimate1}, we obtain
  \begin{equation}\label{rho-bound-estimate2}
    \frac{d}{dt}\norm{\rho}_{L^p}\leq\frac{C(p-1)}{p}(1+\norm{\nabla\dot{u}}_{L^2})\norm{\rho}_{L^p}\leq C(1+\norm{\nabla\dot{u}}_{L^2})\norm{\rho}_{L^p}.
  \end{equation}
  Using \eqref{u-2rd-CNS-u} and Gronwall inequality, and letting $p\rightarrow\infty$, we can get \eqref{rho-bound}, and thus complete the proof of Lemma \ref{lem-rho-bound}.
\end{proof}

Once the crucial uniform upper bound of the density \eqref{rho-bound} is obtained, one can adjust the methods of
\cite{Wang2020} and \cite{Huang2011-2} to prove the following lemma.

\begin{lem}
  Under the assumptions of Theorem \ref{thm-blowup-CNS} and \eqref{assumption3}, it holds
   \begin{equation}\label{u-rho-final}
     \sup_{t\in[0,T]}(\norm{\rho}_{H^1\cap W^{1,q}}+\norm{\nabla u}_{H^1})\leq C,
   \end{equation}
   for any $T\in(0,T^*)$.
\end{lem}

{\bf{Proof of Theorem Theorem \ref{thm-blowup-CNS}.}}
In conclusion,  we prove the uniform energy estimate \eqref{goal-inequality2}.
  Combining the available local existence results stated in Remark \ref{rem-local-existence} with \eqref{goal-inequality2}, it is easy to see that  the strong solution can be extended beyond $T^*$ (see \cite{Huang-Li-Wang2013} for instance), which means $T^*$ is not the maximal existence time, and thus leads to a contradiction.
Therefore, we complete the proof of Theorem \ref{thm-blowup-CNS}.

\medskip

\section*{\bf Data availability}
No data was used for the research described in the article.

\section*{\bf Conflicts of interest}

The authors declare no conflict of interest.

\end{document}